\documentclass{amsart}
\usepackage[utf8]{inputenc}
\usepackage{tikz}
\usepackage{verbatim}
\usepackage{graphicx}
\usepackage[mathcal]{euscript}
\usepackage{float}
\usepackage{cite}
\usepackage[normalem]{ulem}
\usepackage{colortbl}
\usepackage{mathtools}
\usepackage{environ}
 \usetikzlibrary{plotmarks}
\allowdisplaybreaks
\usetikzlibrary{decorations.pathreplacing}

\newtheorem{theorem}{Theorem}[section]
\newtheorem{lemma}[theorem]{Lemma}

\theoremstyle{definition}

\newtheorem{example}[theorem]{Example}

\theoremstyle{remark}

\numberwithin{equation}{section}

\def\Span{\operatorname{span}}

\title{The Jones polynomial of an almost alternating link}

\author{Adam M. Lowrance}
\address{Department of Mathematics and Statistics\\
Vassar College\\
Poughkeepsie, NY} 
\email{adlowrance@vassar.edu}

\author{Dean Spyropoulos}
\address{Department of Mathematics and Statistics\\
Vassar College\\
Poughkeepsie, NY} 
\email{despyropoulos@vassar.edu}

\date{}
\begin{document}
\begin{abstract}
A link is almost alternating if it is non-alternating and has a diagram that can be transformed into an alternating diagram via one crossing change. We give formulas for the first two and last two potential coefficients of the Jones polynomial of an almost alternating link. Using these formulas, we show that the Jones polynomial of an almost alternating link is nontrivial. We also show that either the first two or last two coefficients of the Jones polynomial of an almost alternating link alternate in sign. Finally, we describe conditions that ensure an almost alternating diagram has the fewest number of crossings among all almost alternating diagrams of the link.
\end{abstract}

\maketitle

\section{Introduction}

A link diagram is {\em alternating} if the crossings alternate over, under, over, under, etc. as one traverses each component of the link, and a link is {\em alternating} if it has an alternating diagram. Otherwise, a link is {\em non-alternating}. Alternating links form an important and well-studied class of links. Invariants of alternating links often take on special forms, and the complement of an alternating link has a particularly nice geometric structure \cite{Menasco:AltHyp}. Despite their diagram-dependent definition, alternating knots have recently been shown to have topological characterizations \cite{Greene:Alternating, Howie:Alternating}.

Adams et al. \cite{Adams:Almost} generalized alternating links to the class of almost alternating links.  A link diagram is {\em almost alternating} if one crossing of the diagram can be changed to transform it into an alternating diagram. A link is {\em almost alternating} if it is non-alternating and has an almost alternating diagram. Following the topological characterization of alternating knots, almost alternating knots were also shown to have topological characterizations \cite{Ito:AlmostAlt, Kim:Toroidally}. 

Another generalization of alternating links are links of Turaev genus one. Turaev \cite{Turaev:Jones} gave an alternate proof that the span of the Jones polynomial of a non-split alternating link equals its crossing number. In his proof, he associated to each link diagram an oriented Heegaard surface on which the link has an alternating projection, now known as the {\em Turaev surface} of the link diagram. The genus of the Turaev surface of a connected link diagram $D$ is given by
$$g_T(D) = \frac{1}{2}(2 + c(D) - s_A(D) - s_B(D))$$
where $c(D)$ is the number of crossings in $D$ and $s_A(D)$ and $s_B(D)$ are the number of components in the all-$A$ and all-$B$ Kauffman states of $D$ respectively. The {\em Turaev genus} $g_T(L)$ of a link $L$ is 
$$g_T(L) = \min\{g_T(D)~|~D~\text{is a diagram of}~L\}.$$ Turaev \cite{Turaev:Jones} proved that a link is alternating if and only if it is Turaev genus zero (see also \cite{DFKLS:Jones}). Armond and Lowrance \cite{ArmLow:Turaev} proved that every link of Turaev genus one is mutant to an almost alternating link. In this article, we study the Jones polynomial of a link that is almost alternating or has Turaev genus one.

The Jones polynomial $V_L(t)$ of a link $L$ with $\ell$ components is a Laurent polynomial, first discovered by Jones \cite{Jones:Polynomial}. A fundamental open question about the Jones polynomial is whether it detects the unknot.
Jones \cite{Jones:Ten} conjectured that if a knot has the same Jones polynomial as the unknot, then the knot is the unknot. The Jones unknotting conjecture has been verified in many cases. Kauffman \cite{Kauffman:StateModels}, Murasugi \cite{Murasugi:Jones}, and Thistlethwaite \cite{Thistlethwaite:Jones} proved that the span of the Jones polynomial of a non-split alternating link equals its crossing number, and thus for alternating knots, the Jones unknotting conjecture holds. Lickorish and Thistlethwaite \cite{LT:Adequate} and Stoimenow \cite{Stoimenow:Semi} showed that the Jones unknotting conjecture holds for adequate and semi-adequate knots respectively. Computational results by Hoste, Thistlethwaite, and Weeks \cite{HTW:Tables}; Dasbach and Hougardy \cite{DH:JonesUnknot}; Yamada \cite{Yamada:JonesUnknot}; and Tuzun and Sikora \cite{TS:JonesUnknot} have verified the Jones unknotting conjecture for all knots with at most 22 crossings.

Many authors have developed strategies to produce a nontrivial knot with trivial Jones polynomial, so far without success. Bigelow \cite{Bigelow:JonesUnknot} and Ito \cite{Ito:Braid} proved that if the Burau representation of the four stranded braid group is not faithful, then there exists a nontrivial knot with trivial Jones polynomial. Anstee, Przytycki, and Rolfsen \cite{APR:Mutation}; Jones and Rolfsen \cite{Jones:Mutation};  Rolfsen \cite{Rolfsen:Mutation}; and Przytycki \cite{Przytycki:Mutation} used generalized forms of mutation to attempt to produce nontrivial knots with trivial Jones polynomial. Kauffman \cite{Kauffman:Virtual} produced virtual knots with trivial Jones polynomial, and perhaps one of these examples could be classical. Cohen and Krishnan \cite{CohenKrishnan:Jones} proposed a probabilistic approach for showing that there exists a nontrivial knot with trivial Jones polynomial, and they refined that approach together with Even-Zohar \cite{CEK:Random}.

Although it is an open question whether the Jones polynomial detects the unknot, it is known that for each $\ell\geq 2$, there exist a nontrivial $\ell$-component link whose Jones polynomial equals the Jones polynomial of the $\ell$-component unlink. Thistlethwaite \cite{Thistlethwaite:TrivialLink} found the first examples of nontrivial links with trivial Jones polynomial via a computer tabulation. Eliahou, Kauffman, and Thistlethwaite \cite{EKT:TrivialLink} later generated infinite families of nontrivial $\ell$-component links with trivial Jones polynomials for each $\ell\geq 2$. Despite these examples, for many well-studied classes of links (e.g. alternating, adequate, and semi-adequate links), every nontrivial link has a nontrivial Jones polynomial. 

 The $\ell$-component unlink $\bigcirc \sqcup \cdots \sqcup \bigcirc$ has Jones polynomial $V_{\bigcirc \sqcup \cdots \sqcup \bigcirc}(t) = \left(-t^{\frac{1}{2}}-t^{-\frac{1}{2}}\right)^{\ell-1}$.  Our first main theorem is a slightly stronger version of the statement that every almost alternating or Turaev genus one link has nontrivial Jones polynomial.
\begin{theorem}
\label{theorem:AANontrivial}
Let $L$ be an $\ell$-component almost alternating link or a link of Turaev genus one where $\ell\geq 1$, and let $V_L(t)$ be the Jones polynomial of $L$. Then 
$$V_L(t)\neq t^k \left(-t^{\frac{1}{2}}-t^{-\frac{1}{2}}\right)^{\ell-1}$$
for any $k\in\mathbb{Z}$. In particular, the Jones polynomial of $L$ is different from the Jones polynomial of the $\ell$-component unlink.
\end{theorem}
A consequence of Theorem \ref{theorem:AANontrivial} is that the examples of \cite{EKT:TrivialLink} whose Jones polynomials equal $t^k\left(-t^{1/2}-t^{-1/2}\right)^{\ell-1}$ cannot be almost alternating and are of Turaev genus at least two.

Kauffman \cite{Kauffman:StateModels} proved that the absolute values of the first and last coefficients of the Jones polynomial of an alternating link are one. Dasbach and Lowrance \cite{DasLow:TuraevJones} proved that at least one of the first or last coefficient of the Jones polynomial of an almost alternating or Turaev genus one link has absolute value one. Thistlethwaite \cite{Thistlethwaite:Jones}  proved that the coefficients of the Jones polynomial of a non-split alternating link alternate in sign, that is, the product of consecutive coefficients is at most zero. Our next theorem is a partial generalization of Thistlethwaite's result to almost alternating and Turaev genus one links.
\begin{theorem}
\label{theorem:Sign}
Let $L$ be a non-split almost alternating link or a link of Turaev genus one. Suppose that the Jones polynomial of $L$ is given by
$$V_L(t) = a_0 t^{k} + a_1 t^{k+1} + \cdots + a_{n-1} t^{k+n-1}+a_n t^{k+n}$$
where $n\in\mathbb{Z}_+$, $a_i \in \mathbb{Z}$, $a_0$ and $a_n$ are nonzero, and $k\in \frac{1}{2}\mathbb{Z}$. Either
\begin{itemize}
\item $|a_0|=1$ and $a_0 a_1\leq 0$, or 
\item $|a_n|=1$ and $a_{n-1} a_{n} \leq 0$.
\end{itemize}
\end{theorem}

Dasbach and Lin \cite{DL:Volumish} gave formulas for the second, third, antepenultimate, and penultimate coefficients of the Jones polynomial of an alternating link (see Theorem \ref{theorem:DasLin}). In Theorem \ref{theorem:AAJones}, we apply Dasbach and Lin's result to almost alternating links to obtain a formula for the first and last two potential coefficients of the Jones polynomial. We say the coefficients are potential coefficients since they are potentially zero. Theorem \ref{theorem:AAJones} is the main technical tool used in our proofs of Theorems \ref{theorem:AANontrivial} and \ref{theorem:Sign}. 

Theorem \ref{theorem:Crossing} addresses the question of when an almost alternating diagram has the fewest number of crossings among all almost alternating diagrams of the link. We describe conditions that ensure that an almost alternating diagram has the fewest number of crossings among all almost alternating diagrams of the link. We also describe conditions that place bounds on the fewest number of crossings an almost alternating diagram can have, and those that ensure the link admits an almost alternating diagram with fewer crossings distinct from the one that is given. We leave the complete statement of Theorem \ref{theorem:Crossing} to Section \ref{section:Formula} as the sets of conditions involve quantities obtained from the checkerboard graphs of the diagram that have not yet been defined. The related question, originally asked in \cite{Adams:Almost}, of whether there is an almost alternating diagram $D$ of a link $L$ that has the fewest number of crossings among all diagrams of the link remains open.

This paper is organized as follows. In Section \ref{section:Background}, we recall the construction of the Jones polynomial via the Kauffman bracket and state results on the Jones polynomial of an alternating link. In Section \ref{section:Formula}, we prove Theorem \ref{theorem:AAJones} giving a formula for the potential extreme coefficients of the Jones polynomial of an almost alternating link. We also prove Theorems \ref{theorem:Sign} and \ref{theorem:Crossing}. In Section \ref{section:JonesUnknot}, we prove Theorem \ref{theorem:AANontrivial}, showing that almost alternating and Turaev genus one links have nontrivial Jones polynomials.

\noindent{\bf Acknowledgement.} The authors thank Oliver Dasbach and John McCleary for their thoughts on a draft of this paper.

\section{The Jones polynomial of an alternating link}
\label{section:Background}

In this section, we recall the construction of the Jones polynomial via the Kauffman bracket. We also recall other results related to the Jones polynomial of an alternating link.

The {\em Kauffman bracket} of a link diagram $D$, denoted by $\langle D \rangle$, is a Laurent polynomial with integer coefficients in the formal variable $A$, i.e. $\langle D \rangle \in \mathbb{Z}[A,A^{-1}]$. It is defined recursively by the following three rules:
\begin{enumerate}
\item $\left\langle~
\tikz[baseline=.6ex, scale = .4]{
\draw (0,0) -- (1,1);
\draw (1,0) -- (.7,.3);
\draw (.3,.7) -- (0,1);
}
~\right\rangle = A \left\langle ~ \tikz[baseline=.6ex, scale = .4]{
\draw[rounded corners = 1mm] (0,0) -- (.45,.5) -- (0,1);
\draw[rounded corners = 1mm] (1,0) -- (.55,.5) -- (1,1);
}~\right\rangle + A^{-1} \left\langle~ \tikz[baseline=.6ex, scale = .4]{
\draw[rounded corners = 1mm] (0,0) -- (.5,.45) -- (1,0);
\draw[rounded corners = 1mm] (0,1) -- (.5,.55) -- (1,1);
}~\right\rangle,$
\item $\left\langle~D\sqcup \bigcirc ~\right\rangle = (-A^2 - A^{-2})\left\langle D \right\rangle,$
\item $\left\langle ~ \bigcirc ~\right\rangle = 1.$
\end{enumerate}
In rule (1), the diagram is only changed within a small neighborhood of the pictured crossing. The first term in the sum is the $A$-resolution of the crossing, and the second term is the $B$-resolution of the crossing; see Figure \ref{figure:Resolution}. Rule (2) gives the method for removing a closed component of the diagram without crossings. Finally, rule (3) sets the value of the Kauffman bracket on the unknot.

An alternate formulation of the Kauffman bracket is via the Kauffman state expansion of $D$. A {\em Kauffman state} is the collection of simple closed curves obtained by choosing either an $A$-resolution or a $B$-resolution at each crossing. When performing a resolution, we record where the crossing was with a small line segment called the {\em trace} of the crossing; again see Figure \ref{figure:Resolution}. The trace of a crossing is not considered part of the Kauffman state. For each Kauffman state $S$, define $a(S)$ and $b(S)$ to be the number of $A$-resolutions and the number of $B$-resolutions in $S$ respectively. Define $|S|$ to be the number of components in the Kauffman state $S$. The Kauffman bracket can be expressed as the sum
\begin{equation}
\label{equation:StateSum}
\langle D \rangle = \sum_{S} A^{a(S)-b(S)}\left(-A^2 - A^{-2}\right)^{|S|-1}.
\end{equation}

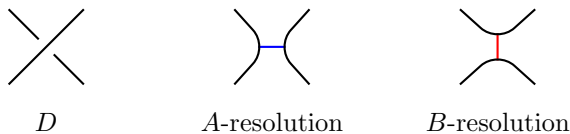
\begin{figure}[h]
$$\begin{tikzpicture}[thick]

\draw (0,0) -- (1,1);
\draw (0,1) -- (.4,.6);
\draw (1,0) -- (.6,.4);
\draw (.5,-.5) node{$D$};

\begin{scope}[xshift = 3cm]
\draw[blue] (.32,.5) -- (.68,.5);
\draw[rounded corners = 3mm] (0,0) -- (.45,.5) -- (0,1);
\draw[rounded corners = 3mm] (1,0) -- (.55,.5) -- (1,1);

\draw (.5,-.5) node{$A$-resolution};
\end{scope}

\begin{scope}[xshift = 6cm]
\draw[red]  (.5,.32) -- (.5,.68);
\draw[rounded corners = 3mm] (0,0) -- (.5,.45) -- (1,0);
\draw[rounded corners = 3mm] (0,1) -- (.5,.55) -- (1,1);

\draw (.5,-.5) node{$B$-resolution};
\end{scope}

\end{tikzpicture}$$
\caption{A crossing in a link diagram $D$ together with its $A$-resolution and $B$-resolution. The trace of the $A$-resolution is the blue line segment, and the trace of the $B$-resolution is the red line segment.}
\label{figure:Resolution}
\end{figure}

Each crossing in an oriented link diagram is either positive $\left(~\tikz[baseline=.6ex, scale = .4]{
\draw[->] (0,0) -- (1,1);
\draw (1,0) -- (.7,.3);
\draw[->] (.3,.7) -- (0,1);
}~\right)$ or negative $\left(~\tikz[baseline=.6ex, scale = .4]{
\draw[->] (.7,.7) -- (1,1);
\draw[->] (1,0) -- (0,1);
\draw (0,0) -- (.3,.3);
}~\right)$. The {\em writhe} $w(D)$ of an oriented link diagram is the difference between the number of positive crossings and the number of negative crossings in $D$. The {\em Jones polynomial} of an oriented link $L$ with diagram $D$ is defined as
$$V_L(t) = \left.(-A^3)^{-w(D)}\langle D \rangle\right|_{A=t^{-1/4}}.$$

Every link diagram $D$ has two dual {\em checkerboard graphs} $G$ and $\overline{G}$. Shade the complementary regions of the link diagram in a checkerboard fashion, i.e. at each crossing the shading should look like $\tikz[baseline=.6ex, scale = .4, thick]{
\fill[black!20!white] (0,0) -- (.5,.5) -- (0,1);
\fill[black!20!white] (1,1) -- (.5,.5) -- (1,0);
\draw (0,0) -- (1,1);
\draw (0,1) -- (.3,.7);
\draw (1,0) -- (.7,.3);
}$ or $\tikz[baseline=.6ex, scale = .4, thick]{
\fill[black!20!white] (0,0) -- (.5,.5) -- (1,0);
\fill[black!20!white] (1,1) -- (.5,.5) -- (0,1);
\draw (0,0) -- (1,1);
\draw (0,1) -- (.3,.7);
\draw (1,0) -- (.7,.3);
}$. The vertices of $G$ are in one-to-one correspondence with the shaded regions of the diagram $D$, and the vertices of $\overline{G}$ are in one-to-one correspondence with the unshaded regions of $D$ (or vice versa). The edges of $G$ are in one-to-one correspondence with the crossings of $D$, and likewise, the edges of $\overline{G}$ are in one-to-one correspondence with the crossings of $D$. An edge $e$ in $G$ (or $\overline{G}$) is incident to vertices $v_1$ and $v_2$ in $G$ (or $\overline{G}$) if the regions associated with $v_1$ and $v_2$ in $D$ meet at the crossing associated to $e$. The graphs $G$ and $\overline{G}$ are planar duals of one another. For an example of the checkerboard graphs of a link diagram see Figure \ref{figure:AAExample}.

An edge in $G$ or $\overline{G}$ is labeled as  $A$-edge or a $B$-edge according to the convention of Figure \ref{figure:CheckEdge}. In an alternating diagram, every edge in $G$ is an $A$-edge and every edge in $\overline{G}$ is a $B$-edge, or vice versa. A link diagram is {\em reduced} if neither of its checkerboard graphs contain any loops.
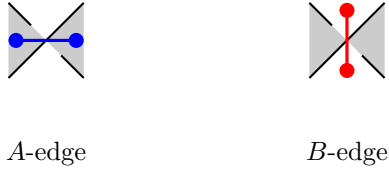
\begin{figure}[h]
$$\begin{tikzpicture}[thick]
\fill[black!20!white] (0,0) -- (.5,.5) -- (0,1);
\fill[black!20!white] (1,1) -- (.5,.5) -- (1,0);
\draw (0,0) -- (1,1);
\draw (0,1) -- (.3,.7);
\draw (1,0) -- (.7,.3);

\draw[very thick, blue] (.1,.5) -- (.9,.5);
\fill[blue] (.1,.5) circle (.1cm);
\fill[blue] (.9,.5) circle (.1cm);
\draw (.5,-1) node{$A$-edge};

\begin{scope}[xshift=4cm]
\fill[black!20!white] (0,0) -- (.5,.5) -- (0,1);
\fill[black!20!white] (1,1) -- (.5,.5) -- (1,0);
\draw (0,0) -- (1,1);
\draw (0,1) -- (.3,.7);
\draw (1,0) -- (.7,.3);

\draw[very thick, red] (.5,.1) -- (.5,.9);
\fill[red] (.5,.1) circle (.1cm);
\fill[red] (.5,.9) circle (.1cm);
\draw (.5,-1) node{$B$-edge};
\end{scope}

\end{tikzpicture}$$
\caption{An $A$-edge and a $B$-edge in the checkerboard graphs.}
\label{figure:CheckEdge}
\end{figure} 

The simplification $G'$ of a graph $G$ is the graph obtained by deleting all loops in $G$ by replacing each set of multiple edges with a single edge. When drawing a graph $G$ or $\overline{G}$ with no loops, our convention is to draw the simplifications $G'$ and $\overline{G}'$. If an edge $e$ in $G'$ or $\overline{G}'$ is incident to vertices $v_1$ and $v_2$, then we label it with the number of edges in $G$ or $\overline{G}$ respectively that are incident to $v_1$ and $v_2$.

 Let $D$ be a reduced alternating diagram. Let $G$ be its checkerboard graph with only $A$-edges, and let $G'$ be the simplification of $G$. Define $v$ and $e$ to be the number of vertices and edges respectively of $G'$. Define $\mu$ to be the number of edges in $G'$ that correspond to multiple edges in $G$, and define $\tau$ to be the number of triangles (or $3$-cycles) in $G'$. Similarly define $\overline{v}$, $\overline{e}$, $\overline{\mu}$, and $\overline{\tau}$ for the checkerboard graph $\overline{G}$ containing only $B$-edges and its simplification $\overline{G}'$. Dasbach and Lin \cite{DL:Volumish} gave formulas for the first three and last three terms in the Kauffman bracket of $D$.
\begin{theorem}[Dasbach, Lin]
\label{theorem:DasLin}
The Kauffman bracket of a reduced alternating diagram $D$ with $c$ crossings is given by
$$\langle D \rangle =  \sum_{i=0}^c \gamma_i A^{c+2v-2-4i}$$
where
\begin{align*}
\gamma_0 = & \; (-1)^{v-1},\\ 
\gamma_1 = & \; (-1)^{v-2}(e-v+1),\\
\gamma_2 = & \; (-1)^{v-3}\left(\binom{v-1}{2} - e(v-2) + \mu + \binom{e}{2} - \tau\right),\\
\gamma_{c-2} = & \; (-1)^{\overline{v}-3}\left(\binom{\overline{v}-1}{2} - \overline{e}(\overline{v}-2) + \overline{\mu} + \binom{\overline{e}}{2} - \overline{\tau}\right),\\
\gamma_{c-1} = & \; (-1)^{\overline{v}-2}(\overline{e}-\overline{v}+1),~\text{and}\\
\gamma_{c} = & \; (-1)^{\overline{v}-1}.
\end{align*}
\end{theorem}

The coefficients $\gamma_0$ and $\gamma_c$ in the above theorem were computed by Kauffman \cite{Kauffman:StateModels}. Dasbach and Lin \cite{DL:Colored} and Stoimenow \cite{Stoimenow:Semi} later extended this theorem to semi-adequate links. In Section \ref{section:Formula}, we use Theorem \ref{theorem:DasLin} to give formulas for some of the coefficients of the Jones polynomial of an almost alternating link.

\section{Jones polynomial formulas}
\label{section:Formula}

In this section, we give formulas for the first two and last two potential coefficients of the Jones polynomial of an almost alternating link. We recall the relationship between almost alternating links and links of Turaev genus one, and we prove that either the first two or last two coefficients of the Jones polynomial of such links alternate in sign.

Let $D$ be an almost alternating diagram as in Figure \ref{figure:AlmostAlternating}. The tangle $R$ is an alternating tangle; when a strand meeting $R$ is labeled either ``$+$" or ``$-$", it indicates that the strand passes over or under respectively another strand in the first crossing involving that strand inside $R$. The diagram obtained by changing the depicted crossing (known as the {\em dealternator}) is alternating. Throughout this section, $G$ will be the checkerboard graph of $D$ containing the vertices $u_1$ and $u_2$, while $\overline{G}$ will be the checkerboard graph of $D$ containing the vertices $v_1$ and $v_2$. Every edge in $G$ except for the edge associated to the dealternator is an $A$-edge, and every edge in $\overline{G}$ except for the edge associated to the dealternator is a $B$-edge. The edges in $G$ and $\overline{G}$ associated to the dealternator will be depicted by a dashed edge.

\begin{figure}[h]
$$\begin{tikzpicture}
\draw[thick, rounded corners = 2mm] (4,0) -- (3.3,.7) -- (1.7,.7) -- (.3,-.7) -- (.3,-1.5) -- (5,-1.5) -- (5,-1) -- (4,0);
\draw[thick, rounded corners = 2mm] (4,0) -- (5,1) -- (5,1.5) -- (.3,1.5) -- (.3, .7) -- (.8,.2);
\draw[thick, rounded corners = 2mm] (1.2,-.2) -- (1.7,-.7) -- (3.3,-.7) -- (4,0);

\fill[white] (4,0) circle (1cm);
\draw[thick] (4,0) circle (1cm);
\draw (4,0) node{$R$};
\draw (3.45,.55) node{$+$};
\draw (4.55,.55) node{$-$};
\draw (4.55,-.55) node{$+$};
\draw (3.45,-.55)node{$-$};

\draw (2.5,1.1) node{\small{$u_1$}};
\draw (2.5,-1.1) node{\small{$u_2$}};
\draw (2.5,0) node{\small{$v_1$}};
\draw (5.5,0) node{\small{$v_2$}};
\end{tikzpicture}$$
\caption{A generic almost alternating diagram.}
\label{figure:AlmostAlternating}
\end{figure}
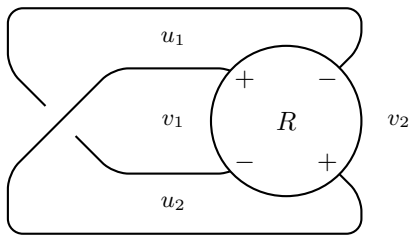

Our goal is to find an expression for the first and last two coefficients of the Jones polynomial of an almost alternating link. Figure \ref{figure:SecretlyAlternating} shows that if $u_1$ and $u_2$ are the same vertex or if $u_1$ and $u_2$ are incident to an edge not associated to the dealternator, then $D$ is an almost alternating diagram of an alternating link. By a symmetric argument, if $v_1$ and $v_2$ are the same or if $v_1$ and $v_2$ are incident to an edge not associated to the dealternator, then $D$ is an almost alternating diagram of an alternating link. Thus we assume that $u_1$ and $u_2$ are distinct, $v_1$ and $v_2$ are distinct, the only edge in $G$ incident to $u_1$ and $u_2$ is associated to the dealternator, and the only edge in $\overline{G}$ incident to $v_1$ and $v_2$ is associated to the dealternator. If an almost alternating diagram satisfies these conditions and if $G$ and $\overline{G}$ do not contain any loops, then we call the diagram $D$ a {\em strongly reduced} almost alternating diagram.
\begin{figure}[h]
$$\begin{tikzpicture}[thick, scale = .8]


	\draw[rounded corners = 2mm] (2,.6) -- (2,.8) -- (1.2,.8) -- (1.2,.3) -- (.6,-.3) -- (.6,-1.2) -- (4,-1.2) -- (4,-.6);
	\draw[rounded corners = 2mm] (2,-.6) -- (2,-.8) -- (1.2,-.8) -- (1.2,-.3) -- (1,-.1);
	\draw[rounded corners = 2mm] (.8,.1) -- (.6,.3) -- (.6,1.2) -- (4,1.2) -- (4,.6);
	
	\fill[white] (2,0) circle (.6cm);
	\draw (2,0) circle (.6cm);
	\draw (2,0) node{$R_1$};
	\draw (2,.4) node{$+$};
	\draw (2,-.4) node{$-$};
	
	\fill[white](4,0) circle (.6cm);
	\draw (4,0) circle (.6cm);
	\draw (4,0) node{$R_2$};
	\draw (4,.4) node{$-$};
	\draw (4,-.4) node{$+$};
	
	\draw[ultra thick, ->] (5,0) -- (6,0);
	

\begin{scope}[xshift = 5cm]

	\draw [rounded corners = 6mm] (2,.6) -- (2,1.2) -- (4,1.2) -- (4,.6);
	\draw [rounded corners = 6mm] (2,-.6) -- (2,-1.2) -- (4,-1.2) -- (4,-.6);

	\fill[white] (2,0) circle (.6cm);
	\draw (2,0) circle (.6cm);
	\draw (2,0) node{\raisebox{\depth}{\scalebox{1}[-1]{$R_1$}}};
	\draw (2,.4) node{$+$};
	\draw (2,-.4) node{$-$};
	
	\fill[white](4,0) circle (.6cm);
	\draw (4,0) circle (.6cm);
	\draw (4,0) node{$R_2$};
	\draw (4,.4) node{$-$};
	\draw (4,-.4) node{$+$};

\end{scope}


\begin{scope}[yshift = -4cm, xshift = -2.5 cm]

	\draw[rounded corners = 1mm] (2,0) -- (1.4,.6)-- (1.2,.3) -- (.6,-.3) -- (.6,-1.2) -- (4.6,-1.2) -- (4.6,-.6) -- (4,0);
	\draw[rounded corners = 1mm] (2,0) -- (1.4,-.6) -- (1.2,-.3) -- (1,-.1);
	\draw[rounded corners = 1mm] (.8,.1) -- (.6,.3) -- (.6,1.2) -- (4.6,1.2) -- (4.6,.6) -- (4,0);
	\draw[rounded corners = 1mm] (2,0) -- (2.5,.6) -- (2.8,.2) -- (3.2,-.2)  -- (3.5,-.6) -- (4,0);
	\draw[rounded corners = 2mm] (2,0) -- (2.5,-.6) -- (2.9,-.1);
	\draw[rounded corners = 1mm] (3.1,.1) -- (3.5,.6) -- (4,0);
	
	\fill[white] (2,0) circle (.6cm);
	\draw (2,0) circle (.6cm);
	\draw (2,0) node{$R_1$};
	\draw (1.75,.35) node{$+$};
	\draw (1.75,-.35) node{$-$};
	\draw (2.25,.35) node{$-$};
	\draw (2.25,-.35) node{$+$};
	
	\fill[white](4,0) circle (.6cm);
	\draw (4,0) circle (.6cm);
	\draw (4,0) node{$R_2$};
	\draw (3.75,.35) node{$+$};
	\draw (3.75,-.35) node{$-$};
	\draw (4.25,.35) node{$-$};
	\draw (4.25,-.35) node{$+$};

\end{scope}

\draw[ultra thick, ->] (2.3,-4) -- (3.3,-4);
\draw[ultra thick, ->] (8,-4) -- (9,-4);


\begin{scope}[yshift = -4cm, xshift = 3cm]

	\draw[rounded corners = 5mm] (2.5,0) -- (3.25,.8) -- (4,0);
	\draw[rounded corners = 5mm] (2.5,0) -- (3.25, -.8) -- (4,0);
	
	\draw[rounded corners = 1mm] (2.5,0) --(2.1,-.6) -- (1.8,-.3) -- (1.2,.3) -- (.6,-.3) -- (.6,-1.2) -- (4.6,-1.2) -- (4.6,-.6) -- (4,0);
	\draw[rounded corners = 1mm] (.8,.1) -- (.6,.3) -- (.6,1.2) -- (4.6,1.2) -- (4.6,.6) -- (4,0);
	\draw[rounded corners = 1mm] (1,-.1) -- (1.2,-.3) -- (1.4,-.1);
	\draw[rounded corners = 1mm] (2.5,0) --(2.1,.6) -- (1.8,.3) -- (1.5,.1);

	\fill[white] (2.5,0) circle (.6cm);
	\draw (2.5,0) circle (.6cm);
	\draw (2.5,0) node{\raisebox{\depth}{\scalebox{1}[-1]{$R_1$}}};
	\draw (2.75,.35) node{$-$};
	\draw (2.75,-.35) node{$+$};
	\draw (2.25,.35) node{$+$};
	\draw (2.25,-.35) node{$-$};
	
	\fill[white](4,0) circle (.6cm);
	\draw (4,0) circle (.6cm);
	\draw (4,0) node{$R_2$};
	\draw (3.75,.35) node{$+$};
	\draw (3.75,-.35) node{$-$};
	\draw (4.25,.35) node{$-$};
	\draw (4.25,-.35) node{$+$};

\end{scope}


\begin{scope}[yshift = -4cm, xshift = 7.5cm]

	\draw[rounded corners = 5mm] (2.5,0) -- (3.25,.8) -- (4,0);
	\draw[rounded corners = 5mm] (2.5,0) -- (3.25, -.8) -- (4,0);
	\draw[rounded corners = 3mm] (2.5,0) -- (2,.5) -- (2,1.2) -- (4.5,1.2) -- (4.5,.5) -- (4,0);
	\draw[rounded corners = 3mm] (2.5,0) -- (2,-.5) -- (2,-1.2) -- (4.5,-1.2) -- (4.5,-.5) -- (4,0);

	\fill[white] (2.5,0) circle (.6cm);
	\draw (2.5,0) circle (.6cm);
	\draw (2.5,0) node{\raisebox{\depth}{\scalebox{1}[-1]{$R_1$}}};
	\draw (2.75,.35) node{$-$};
	\draw (2.75,-.35) node{$+$};
	\draw (2.25,.35) node{$+$};
	\draw (2.25,-.35) node{$-$};
	
	\fill[white](4,0) circle (.6cm);
	\draw (4,0) circle (.6cm);
	\draw (4,0) node{$R_2$};
	\draw (3.75,.35) node{$+$};
	\draw (3.75,-.35) node{$-$};
	\draw (4.25,.35) node{$-$};
	\draw (4.25,-.35) node{$+$};

\end{scope}

\end{tikzpicture}$$
\caption{ \textbf{Top:} If $u_1=u_2$, then a Reidemeister 1 move transforms $D$ into an alternating diagram. \textbf{ Bottom:} If $u_1$ and $u_2$ are incident to an edge not associated to the dealternator, then a flype and a Reidemeister 2 move transform $D$ into an alternating diagram.}
\label{figure:SecretlyAlternating}
\end{figure}
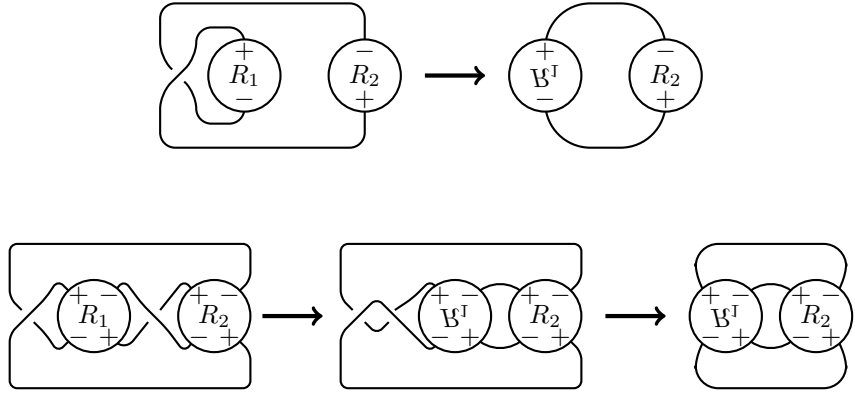

The formulas for the first two and last two coefficients of the Jones polynomial of an almost alternating link make heavy use of the checkerboard graphs $G$ and $\overline{G}$ of $D$. Define $G'$ and $\overline{G}'$ to be the simplifications of $G$ and $\overline{G}$ respectively. Denote the number of vertices and edges in $G'$ and $\overline{G}'$ by $v$, $\overline{v}$, $e$, and $\overline{e}$ respectively. The {\em circuit rank} of a graph is the number of edges not contained in a maximal spanning forest of the graph, or alternatively, it is the first Betti number of the graph when thought of as a cellular complex. Let $\beta_1$ and $\overline{\beta}_1$ be the circuit ranks of $G'$ and $\overline{G}'$ respectively. Since $G'$ and $\overline{G}'$ are connected, we have
$$\beta_1 = e - v + 1~\text{and}~\overline{\beta}_1 = \overline{e}-\overline{v}+1.$$

Let $P$ denote the number of paths of length two between $u_1$ and $u_2$ in $G'$, and let $\overline{P}$ denote the number of paths of length two between $v_1$ and $v_2$ in $\overline{G}'$. For $i=0,1$, and $2$, define $P_i$ to be the number of paths of length two between $u_1$ and $u_2$ such that $i$ of the edges in the path came from multiple edges in $G$. Similarly define $\overline{P}_i$ for $i=0,1$, and $2$. Let $Q$ be the number of paths of length three between $u_1$ and $u_2$ in $G'$ such that no interior vertex of the path is adjacent to both $u_1$ and $u_2$. Similarly, define $\overline{Q}$ to be the number of paths of length three between $v_1$ and $v_2$ in $\overline{G}'$ such that no interior vertex of the path is adjacent to both $v_1$ and $v_2$. Finally, define $S$ to be the number of subgraphs of $G'$ containing $u_1$ and $u_2$ that are isomorphic to the complete graph $K_4$ on four vertices, and similarly define $\overline{S}$ to be the number of subgraphs of $\overline{G}'$ containing $v_1$ and $v_2$ that are isomorphic to $K_4$. See Figure \ref{figure:Square} for depictions of $S$ and $\overline{S}$.
\begin{figure}[h]
$$\begin{tikzpicture}
\begin{scope}[xshift = 5cm]
\draw (-1,1.1) node[left]{\small{$v_1$}};
\draw (1,1.1) node[right]{\small{$v_2$}};
\draw (0,-1) node[below]{$\overline{S}$};

\fill (0,0) circle (.1cm);
\fill (1,1) circle (.1cm);
\fill (-1,1) circle (.1cm);
\fill (0,2) circle (.1cm);
\draw (0,0) -- (1,1) -- (0,2) -- (-1,1) -- (0,0);
\draw (0,0) -- (0,2);
\draw[dashed] (-1,1) arc (90:270:.75cm);
\draw[dashed] (1,1) arc (90:-90:.75cm);
\draw[dashed] (-1,-.5) -- (1,-.5);

\end{scope}

\draw (0.2,0) node[below]{\small{$u_2$}};
\draw (0.2,2) node[above]{\small{$u_1$}};
\draw (0,-1) node[below]{$S$};

\fill (0,0) circle (.1cm);
\fill (1,1) circle (.1cm);
\fill (-1,1) circle (.1cm);
\fill (0,2) circle (.1cm);
\draw (0,0) -- (1,1) -- (0,2) -- (-1,1) -- (0,0);
\draw (-1,1) -- (1,1);
\draw[dashed] (0,2) arc (0:180:.75cm);
\draw[dashed]  (0,0) arc (0:-180:.75cm);
\draw[dashed]  (-1.5,0) -- (-1.5,2);

\end{tikzpicture}$$
\caption{$S$ and $\overline{S}$ are the counts of $K_4$ subgraphs of $G'$ and $\overline{G}'$ respectively. The dashed edge indicates the edge that is associated to the dealternator.}
\label{figure:Square}
\end{figure}
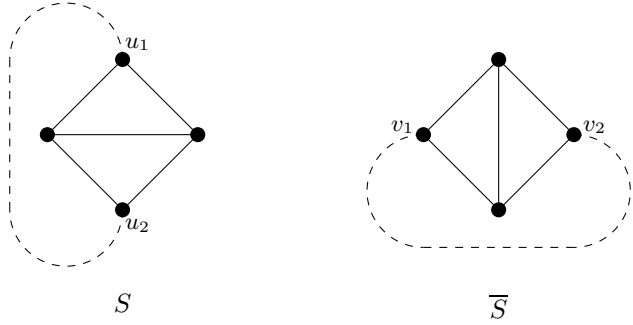

The first and last two potential coefficients of the Jones polynomial of an almost alternating link can be expressed using the above notation. The first and last coefficients were computed in \cite{DasLow:TuraevJones}. The fact that the span of the Jones polynomial of an almost alternating link is at most $c-3$ was proved in \cite{Adams:Almost}. The main contribution of this theorem are the formulas for the second and penultimate coefficients $\alpha_1$ and $\alpha_{c-4}$.
\begin{theorem}
\label{theorem:AAJones}
Let $L$ be a link with strongly reduced almost alternating diagram $D$ with $c$ crossings as in Figure \ref{figure:AlmostAlternating}. The Kauffman bracket of $D$ can be expressed as
$$\langle D \rangle = \sum_{i=0}^{c-3} \alpha_i A^{c+2v-8-4i}$$
where
\begin{align*}
\alpha_0 = &\; (-1)^v(P-1),\\
\alpha_1 = & \; (-1)^{v-1}\left(\beta_1(P-1) - \binom{P}{2} + P_2- P_0 + Q - S\right),\\
\alpha_{c-4} = &  \; (-1)^{\overline{v}-1}\left(\overline{\beta_1}(\overline{P}-1)-\binom{\overline{P}}{2} + \overline{P}_2-\overline{P}_0+\overline{Q}-\overline{S}\right),~\text{and}\\
\alpha_{c-3} = & \; (-1)^{\overline{v}}(\overline{P}-1).\\
\end{align*}
\end{theorem}
\begin{proof}
Let $D$ be an almost alternating diagram as in Figure \ref{figure:AlmostAlternating}, and let $D_A$ and $D_B$ be the $A$ and $B$ resolutions of $D$ at the dealternator. Then $D_A$ and $D_B$ are reduced alternating diagrams where $D_A$ is the denominator closure of $R$ and $D_B$ is the numerator closure of $R$. We prove the formulas for $\alpha_0$ and $\alpha_1$. The proofs $\alpha_{c-4}$ and $\alpha_{c-3}$ are obtained by considering the mirror image of $D$.

Let $c_A, v_A, e_A, \mu_A$, and $\tau_A$ be the terms in Theorem \ref{theorem:DasLin} associated to the first three coefficients of $\langle D_A \rangle$, and similarly let $c_B, v_B, e_B, \mu_B$, and $\tau_B$ be the terms in Theorem \ref{theorem:DasLin} associated to the first three coefficients of $\langle D_B \rangle$. Define $G_A$ and $G_A'$ to be the all-$A$ checkerboard graph of $D_A$ and its simplification, and similarly define $\widetilde{G_A}$ and $\widetilde{G_A}'$ to be the all-$A$ checkerboard graph of $D_B$ and its simplification. 

The graph $G_A$ is obtained from $\widetilde{G_A}$ by identifying the vertices $u_1$ and $u_2$ into a single vertex $u_{12}$. Also, the graph $G$ is obtained from $\widetilde{G_A}$ by adding an edge between the vertices $u_1$ and $u_2$. Therefore $v=v_B=v_A+1$ and $e=e_B+1$. 

The Kauffman bracket of $D$ is computed as 
$$\langle D \rangle = A \langle D_A \rangle + A^{-1} \langle D_B \rangle.$$
We use Theorem \ref{theorem:DasLin} to compute $\langle D_A\rangle$ and $\langle D_B \rangle$. The top degree term of $A \langle D_A \rangle$ is 
$$(-1)^{v_A-1}A^{c_A+2v_A-1} = (-1)^{v-2}A^{c+2v-4}.$$
The top degree term of $A^{-1}\langle D_B\rangle$ is
$$(-1)^{v_B-1}A^{c_B+2v_B-1} = (-1)^{v-1}A^{c+2v-4}.$$
These terms cancel in $\langle D \rangle$, and so the coefficient of $A^{c+2v-4}$ in $\langle D \rangle$ is zero.

The penultimate terms in $A \langle D_A\rangle$ and $A^{-1}\langle D_B \rangle$ are respectively given by
\begin{align*}
(-1)^{v_A-2}(e_A - v_A+1)A^{c_A+2v_A-5} = & \; (-1)^{v-3}(e_A-v+2)A^{c+2v-8}~\text{and}\\
(-1)^{v_B-2}(e_B-v_B+1)A^{c_B+2v_B-7} =& \; (-1)^{v-2}(e_B-v+1)A^{c+2v-8}.
\end{align*}
Recall that $P$ is the number of paths of length two in $G'$ between the vertices $u_1$ and $u_2$. Since a path of length two in $\widetilde{G_A}$ between $u_1$ and $u_2$ becomes a multiple edge in $G_A$, it follows that $e_B-e_A= P$. Thus, the coefficient $\alpha_0$ of $A^{c+2v-8}$ in $\langle D \rangle$ is 
$$\alpha_0=(-1)^{v-2}((e_B - v + 1) - (e_A - v + 2)) = (-1)^v (P-1).$$

The antepenultimate terms in $A\langle D_A \rangle$ and $A^{-1} \langle D_B \rangle$ are given by
\begin{align*}
& \; (-1)^{v_A-3}\left( \binom{v_A-1}{2} - e_A(v_A-2) + \binom{e_A}{2} + \mu_A - \tau_A\right)A^{c_A+2v_A-9}\\
 = & \; (-1)^{v-4}\left( \binom{v-2}{2} -e_A(v-3) + \binom{e_A}{2} + \mu_A - \tau_A\right)A^{c+2v-12}~\text{and}\\ 
 & \; (-1)^{v_B-3}\left( \binom{v_B-1}{2} - e_B(v_B-2) + \binom{e_B}{2} + \mu_B - \tau_B\right)A^{c_B+2v_B-11}\\
 = & \; (-1)^{v-3}\left( \binom{v-1}{2} - e_B(v-2) + \binom{e_B}{2} +\mu_B - \tau_B\right)A^{c+2v-12}.
\end{align*}
Therefore the coefficient of $A^{c+2v-12}$ in $\langle D \rangle$ is 
\begin{align*}
\alpha_1 =  (-1)^{v-1} & \left(  \binom{v-1}{2} - \binom{v-2}{2} +e_A(v-3) - e_B(v-2)\right. \\
& \left. \;+\binom{e_B}{2} - \binom{e_A}{2} +\mu_B - \mu_A + \tau_A - \tau_B\right).
\end{align*}
We handle the terms in $\alpha_1$ in pairs. A straightforward computation shows that
\begin{equation}
\label{equation:term1}
\binom{v-1}{2} - \binom{v-2}{2} = v-2.
\end{equation}
The second pair of terms yields
\begin{align}
\label{equation:term2}
\begin{split}
e_A(v-3) - e_B(v-2) = & \; (e_A-e_B)(v-2)-e_A\\
= & -P(v-2) - e_A.
\end{split}
\end{align}
The third pair of terms yields
\begin{align}
\label{equation:term3}
\begin{split}
\binom{e_B}{2} - \binom{e_A}{2} = & \; \frac{1}{2} (e_B^2 - e_B - e_A^2 + e_A)\\
= & \; \frac{1}{2}((e_B-e_A)(e_B+e_A)-(e_B-e_A))\\
= & \; \frac{1}{2} P(P+2e_A-1)\\
= & \; \frac{1}{2} P^2 - \frac{1}{2} P +Pe_A.
\end{split}
\end{align}

Recall that $\mu_A$ and $\mu_B$ count the number of edges in $G_A'$ and $\widetilde{G_A}'$ respectively that came from multiple edges in $G_A$ and $\widetilde{G_A}$, and recall that $G_A$ is obtained from $\widetilde{G_A}$ by identifying the vertices $u_1$ and $u_2$. Suppose that $e$ is an edge in $\widetilde{G_A}'$ that is not contained in a path of length two between $u_1$ and $u_2$. Then $e$ is also an edge in $\widetilde{G_A}'$ and $e$ came from a multiple edge in $G_A$ if and only if $e$ came from a multiple edge in $\widetilde{G_A}$. Now suppose that $e$ is an edge in $\widetilde{G_A}'$ that is contained in a path of length two between $u_1$ and $u_2$. If both edges in that path came from multiple edges in $\widetilde{G_A}$, then that path contributes two to $\mu_B$ and one to $\mu_A$. If exactly one edge in that path came from multiple edges in $\widetilde{G_A}$, then that path contributes one to $\mu_B$ and one to $\mu_A$. If no edges in the path came from multiple edges in $\widetilde{G_A}$, then the path contributes zero to $\mu_B$ and one to $\mu_A$. Hence
\begin{equation}
\label{equation:term4}
\mu_B - \mu_A = P_2 - P_0,
\end{equation}
where $P_i$ is the number of paths of length two in $G'$ between $u_1$ and $u_2$ such that $i$ edges in the path come from multiple edges in $G$.

To analyze the final pair of terms in $\alpha_1$ we must find the difference in the number of triangles in $G_A'$ and $\widetilde{G_A}'$. The triangles in $G_A'$ either come from a triangle in $\widetilde{G_A}'$ or from a path of length three between $u_1$ and $u_2$ in $\widetilde{G_A}'$. If $T$ is a triangle in $\widetilde{G_A}'$ not containing the vertices $u_1$ or $u_2$, then $T$ is also a triangle in $G_A'$. Suppose $T$ is a triangle in $\widetilde{G_A}'$ containing one of the vertices $u_1$ and $u_2$, say $u_1$. Let $u_3$ and $u_4$ be the other vertices in $T$. If neither $u_3$ nor $u_4$ are adjacent to $u_2$, then the triangle consisting of the vertices $u_{12}$, $u_3$, and $u_4$ in $G_A'$ comes from $T$ and $T$ alone. If exactly one of $u_3$ or $u_4$ is adjacent to $u_2$, then the triangle consisting of the vertices $u_{12}$, $u_3$, and $u_4$ in $G_A'$ comes from the triangle $T$ and the path of length three containing the vertices $u_1, u_2, u_3$, and $u_4$. If both $u_3$ and $u_4$ are adjacent to $u_2$, then there is a triangle $T'$ formed by the vertices $u_2, u_3$, and $u_4$ in $\widetilde{G_A}'$. The triangle in $G_A'$ consisting of the vertices $u_{12}, u_3$, and $u_4$ comes from $T$ and $T'$ in $\widetilde{G_A}'$. Finally, let $H$ be a path of length three between $u_1$ and $u_2$ in $\widetilde{G_A}'$ such that no interior vertex of $H$ is adjacent to both $u_1$ and $u_2$. After identifying $u_1$ and $u_2$ (and simplifying) to form $G_A'$, the path $H$ becomes a triangle. See Figure \ref{figure:Triangle}. Therefore \begin{equation}
\label{equation:term5}
\tau_A-\tau_B = Q - S,
\end{equation}
where $Q$ is the number of paths of length three between $u_1$ and $u_2$ in $G'$ such that no interior vertex of the path is adjacent to both $u_1$ and $u_2$ and $S$ is the number of $K_4$ subgraphs of $G'$ containing $u_1$ and $u_2$ (see Figure \ref{figure:Square}).

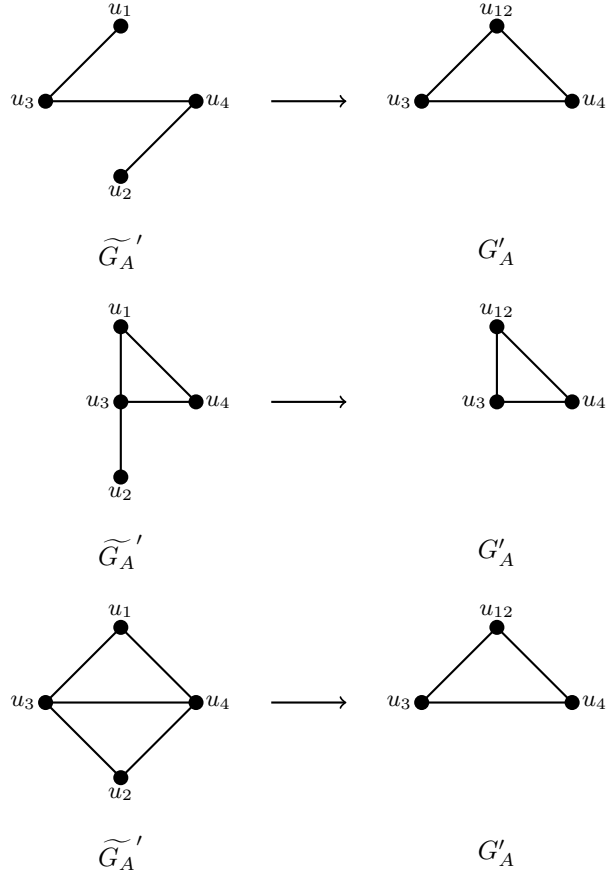
\begin{figure}[h]
$$\begin{tikzpicture}[thick]
\begin{scope}[yshift = 4cm]

	\fill (0,0) circle (.1cm);
	\fill (-1,1) circle (.1cm);
	\fill (1,1) circle (.1cm);
	\fill (0,2) circle (.1cm);
	\draw (0,2) -- (-1,1) -- (1,1) -- (0,0);
	\draw (0,2) node[above]{\small{$u_1$}};
	\draw (0,0) node[below]{\small{$u_2$}};
	\draw (-1,1) node[left]{\small{$u_3$}};
	\draw (1,1) node[right]{\small{$u_4$}};
	\draw (0,-1) node{$\widetilde{G_A}'$};
	
	\begin{scope}[xshift = 5cm]
	\fill (-1,1) circle (.1cm);
	\fill (1,1) circle (.1cm);
	\fill (0,2) circle (.1cm);
	\draw (-1,1) -- (0,2) -- (1,1) -- (-1,1);
	\draw (0,2) node[above]{\small{$u_{12}$}};
	\draw (-1,1) node[left]{\small{$u_3$}};
	\draw (1,1) node[right]{\small{$u_4$}};
	\draw (0,-1) node{$G_A'$};
	\end{scope}

	\draw[->] (2,1) -- (3,1);

\end{scope}


	\fill (0,0) circle (.1cm);
	\fill (0,1) circle (.1cm);
	\fill (1,1) circle (.1cm);
	\fill (0,2) circle (.1cm);
	\draw (0,0) -- (0,2) -- (1,1) -- (0,1);
	\draw (0,2) node[above]{\small{$u_1$}};
	\draw (0,0) node[below]{\small{$u_2$}};
	\draw (0,1) node[left]{\small{$u_3$}};
	\draw (1,1) node[right]{\small{$u_4$}};
	\draw (0,-1) node{$\widetilde{G_A}'$};

	\begin{scope}[xshift = 5cm]
	\fill (0,1) circle (.1cm);
	\fill (1,1) circle (.1cm);
	\fill (0,2) circle (.1cm);
	\draw (0,1) -- (0,2) -- (1,1) -- (0,1);
	\draw (0,2) node[above]{\small{$u_{12}$}};
	\draw (0,1) node[left]{\small{$u_3$}};
	\draw (1,1) node[right]{\small{$u_4$}};
	\draw (0,-1) node{$G_A'$};
	\end{scope}
	
	\draw[->] (2,1) -- (3,1);
	
\begin{scope}[yshift = -4cm]

	\fill (0,0) circle (.1cm);
	\fill (-1,1) circle (.1cm);
	\fill (1,1) circle (.1cm);
	\fill (0,2) circle (.1cm);
	\draw (-1,1) -- (0,2) -- (1,1) -- (0,0) -- (-1,1) -- (1,1);
	\draw (0,2) node[above]{\small{$u_1$}};
	\draw (0,0) node[below]{\small{$u_2$}};
	\draw (-1,1) node[left]{\small{$u_3$}};
	\draw (1,1) node[right]{\small{$u_4$}};
	\draw (0,-1) node{$\widetilde{G_A}'$};
	
	\begin{scope}[xshift = 5cm]
	\fill (-1,1) circle (.1cm);
	\fill (1,1) circle (.1cm);
	\fill (0,2) circle (.1cm);
	\draw (-1,1) -- (0,2) -- (1,1) -- (-1,1);
	\draw (0,2) node[above]{\small{$u_{12}$}};
	\draw (-1,1) node[left]{\small{$u_3$}};
	\draw (1,1) node[right]{\small{$u_4$}};
	\draw (0,-1) node{$G_A'$};
	\end{scope}
	
	\draw[->] (2,1) -- (3,1);

\end{scope}

\end{tikzpicture}$$
\caption{\textbf{Top.} A path of length three such that no interior vertex is adjacent to both endpoints in $\widetilde{G_A}'$ adds a new a triangle in $G_A'$. \textbf{Middle.} A path of length three such that one interior vertex is adjacent to both endpoints in $\widetilde{G_A}'$ preserves the number of triangles in $G_A'$. \textbf{Bottom.} A path of length three such that both interior vertices are adjacent to both endpoints in $\widetilde{G_A}'$ decreases the number of triangles by one in $G_A'$.}
\label{figure:Triangle}
\end{figure}

Combining Equations \ref{equation:term1} through \ref{equation:term5} and ignoring an overall sign yields
\begin{align*}
\alpha_1 = & \; v-2 - P(v-2) -e_A +  \frac{1}{2} P^2 - \frac{1}{2} P +Pe_A + P_2 - P_0 + Q - S\\\
= & \;  (1-P)(v-2 - e_B + P)+ \frac{1}{2} P^2 - \frac{1}{2} P+ P_2 - P_0 + Q - S\\
= & \;  (1-P)(v-1-e) +P - P^2 + \frac{1}{2} P^2 - \frac{1}{2} P+ P_2 - P_0 + Q - S\\
= & \;  \beta_1(P-1) - \binom{P}{2} + P_2- P_0 + Q - S,
\end{align*}
as desired.

It remains to show that the difference in exponent of the $\alpha_0$ term and the $\alpha_{c-3}$ term is $4(c-3)$. The diagram $D$ can be considered as a $4$-regular planar graph whose vertices correspond to the crossings. The number of vertices, edges, and faces of the link diagram considered as a graph are $c$, $2c$, and $v+\overline{v}$ respectively. Since this graph is planar, its Euler characteristic is two, and $v+\overline{v}=c+2$. The difference in exponent between the first and last potential terms of $\langle D \rangle$ is
\begin{align*}
(c+2v - 8) - (-c-2\overline{v} + 8) = &\; 2c + 2(v+\overline{v}) -16\\
= & \; 2c + 2(c+2) - 16\\
= & \;4(c-3),
\end{align*}
as desired.
\end{proof}

\begin{example}
Let $L$ be the two-component almost alternating link with diagram $D$ as in Figure \ref{figure:AAExample}. For this diagram, we have $v=7, P=P_0=P_1=P_2=S=0, Q=1,\beta_1=4,\overline{v}=5, \overline{P}=3, \overline{P_0}=1, \overline{P_1}=2, \overline{P_2}=\overline{Q}=0, \overline{S}=1$, and $\overline{\beta}_1=4$. Theorem \ref{theorem:AAJones} implies that
$$\alpha_0 = 1,~\alpha_1=-3,~\alpha_{c-4}=3,~\text{and}~\alpha_{c-3}=-2.$$
The values of the coefficients can be seen in the Jones polynomial of $L$:
$$V_L(t) = t^{-17/2} - 3t^{-15/2} + 4t^{-13/2} - 5t^{-11/2} + 5t^{-9/2} - 5t^{-7/2} + 3t^{-5/2} -2t^{-3/2}.$$
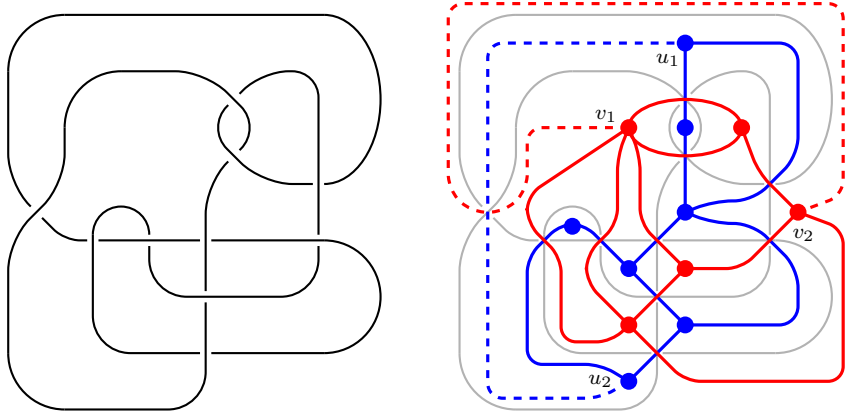
\begin{figure}[h]
$$\begin{tikzpicture}[thick, scale = .75]


\draw[rounded corners=4mm] (.4,4.6) -- (0,5) -- (0,7);
\draw (0,7) arc (180:90:1cm);
\draw[rounded corners = 2mm] (.6,4.4) -- (1,4) -- (1.4,4);
\draw (1.6,4) -- (3.4,4);
\draw (2.5,4.1) arc (0:180:.5cm);
\draw[rounded corners = 5mm] (1.5,4.1) -- (1.5,2) -- (3.4,2);
\draw[rounded corners = 5mm] (2.5,3.9) -- (2.5,3) -- (5.5,3) -- (5.5,3.9);
\draw[rounded corners = 4mm] (3.5,3.1) -- (3.5,5) -- (3.9,5.4);
\draw[rounded corners = 4mm] (3.9,6.4) -- (3.5,6) -- (4.5,5) -- (5.4,5);
\draw[rounded corners = 4mm] (4.1,5.6) -- (4.5,6) -- (3.5,7) -- (2,7);
\draw (2,7) arc (90:180:1cm);
\draw[rounded corners = 4mm] (1,6) -- (1,5) -- (0,4) -- (0,2);
\draw (0,2) arc (180:270:1cm);
\draw[rounded corners = 5mm] (1,1) -- (3.5,1) -- (3.5,2.9);
\draw[rounded corners = 3.5mm] (4.1,6.6) -- (4.5,7) -- (5.5,7) -- (5.5,4.1);
\draw (3.6,4) -- (5.6,4);
\draw (5.6,4) arc (90:-90:1cm);
\draw (3.6,2) -- (5.6,2);
\draw (1,8) --(5.6,8);
\draw (5.6,8) arc (90:-90:1cm and 1.5cm);


\begin{scope}[xshift = 8cm, white!70!black]

	\draw[rounded corners=4mm] (.4,4.6) -- (0,5) -- (0,7);
	\draw (0,7) arc (180:90:1cm);
	\draw[rounded corners = 2mm] (.6,4.4) -- (1,4) -- (1.4,4);
	\draw (1.6,4) -- (3.4,4);
	\draw (2.5,4.1) arc (0:180:.5cm);
	\draw[rounded corners = 5mm] (1.5,4.1) -- (1.5,2) -- (3.4,2);
	\draw[rounded corners = 5mm] (2.5,3.9) -- (2.5,3) -- (5.5,3) -- (5.5,3.9);
	\draw[rounded corners = 4mm] (3.5,3.1) -- (3.5,5) -- (3.9,5.4);
	\draw[rounded corners = 4mm] (3.9,6.4) -- (3.5,6) -- (4.5,5) -- (5.4,5);
	\draw[rounded corners = 4mm] (4.1,5.6) -- (4.5,6) -- (3.5,7) -- (2,7);
	\draw (2,7) arc (90:180:1cm);
	\draw[rounded corners = 4mm] (1,6) -- (1,5) -- (0,4) -- (0,2);
	\draw (0,2) arc (180:270:1cm);
	\draw[rounded corners = 5mm] (1,1) -- (3.5,1) -- (3.5,2.9);
	\draw[rounded corners = 3.5mm] (4.1,6.6) -- (4.5,7) -- (5.5,7) -- (5.5,4.1);
	\draw (3.6,4) -- (5.6,4);	
	\draw (5.6,4) arc (90:-90:1cm);
	\draw (3.6,2) -- (5.6,2);
	\draw (1,8) --(5.6,8);
	\draw (5.6,8) arc (90:-90:1cm and 1.5cm);
	
	
	\begin{scope}[blue, very thick]
		
		\fill (4,7.5) circle (.15cm);
		\fill (4,6) circle (.15cm);
		\fill (4,4.5) circle (.15cm);
		\fill (3,3.5) circle (.15cm);
		\fill (4,2.5) circle (.15cm);
		\fill (2,4.25) circle (.15cm);
		\fill (3,1.5) circle (.15cm);
		
		\draw[dashed, rounded corners = 2mm] (3,1.5) -- (2.5,1.2)  -- (.5,1.2) -- (.5,7.5) -- (4,7.5);
		\draw (3,1.5) -- (4,2.5) -- (3,3.5) -- (4,4.5) -- (4,6) -- (4,7.5);
		\draw[rounded corners = 2mm] (3,1.5) -- (2.5,1.8) -- (1.2,1.8) -- (1.2,3.7) -- (1.8,4.3) -- (2.2,4.3) -- (2.8,3.7) -- (3,3.5);
		\draw[rounded corners = 2.5mm] (4,4.5) -- (4.5,4.3) -- (5.2,4.3) -- (6,3.5) -- (6,2.5) -- (4,2.5);
		\draw[rounded corners = 2.5mm] (4,4.5) -- (4.5,4.7) -- (5.2,4.7) -- (6,5.5) -- (6,7.5) -- (4,7.5);
	
	\end{scope}
	
	
	\begin{scope}[red, very thick]
	
	\fill (3,6) circle (.15cm);
	\fill (4,3.5) circle (.15cm);
	\fill (3,2.5) circle (.15cm);
	\fill (6,4.5) circle (.15cm);
	\fill (5,6) circle (.15cm);
	
	\draw[dashed, rounded corners = 2mm] (3,6) -- (1.2,6) -- (1.2,4.9) -- (.8,4.5) -- (.2,4.5) -- (-.2,4.9) -- (-.2,8.2) -- (6.8,8.2) -- (6.8,4.8) -- (6,4.5);
	\draw (4,6) ellipse (1cm and .5cm);
	\draw (3,2.5) -- (4,3.5);
	\draw[rounded corners = 2mm] (3,2.5) -- (2.25,3.25) -- (2.25,3.75) -- (2.8,4.3) -- (2.8,5.5) -- (3,6);
	\draw[rounded corners = 2mm] (3,2.5) -- (2.7,2.2) -- (1.8,2.2) -- (1.8,3.7) -- (1.2,4.3) -- (1.2,4.8) -- (3,6);
	\draw[rounded corners = 2mm] (3,6) -- (3.2,5.5) -- (3.2,4.3) -- (4,3.5);
	\draw[rounded corners = 2mm] (6,4.5) -- (5.3,5.2) -- (5,6);
	\draw[rounded corners = 2mm] (6,4.5) -- (5,3.5) -- (4,3.5);
	\draw[rounded corners = 2mm] (6,4.5) -- (6.8,4.2) -- (6.8, 1.5) -- (4,1.5) -- (3,2.5);
		
	\end{scope}
	
	\begin{scope}[black]
	\draw (3.7,7.2) node{\footnotesize{$u_1$}};
	\draw (2.5,1.5) node{\footnotesize{$u_2$}};
	\draw (2.6,6.2) node{\footnotesize{$v_1$}};
	\draw (6.1,4.1)node{\footnotesize{$v_2$}};

	\end{scope}
\end{scope}

\end{tikzpicture}$$
\caption{A diagram of an almost alternating link $L$, and its checkerboard graphs $G$ (in blue) and $\overline{G}$ (in red).}
\label{figure:AAExample}
\end{figure}

\end{example}

Adams et al. \cite{Adams:Almost} showed that the span of the Jones polynomial gives a lower bound on the fewest number of crossings in an almost alternating diagram of an almost alternating link, as reflected in Theorem \ref{theorem:AAJones}. If $\Span V_L(t) = c(D)-3$, then $D$ has the fewest number of crossings among all almost alternating diagrams of $L$. In the following theorem, we use the notation from Theorem \ref{theorem:AAJones}. Parts (1) and (4) are implicit, but never directly stated, in \cite{DasLow:TuraevJones}.
\begin{theorem}
\label{theorem:Crossing}
Let $D$ be a strongly reduced almost alternating diagram of the link $L$ with $c(D)$ crossings, as in Figure \ref{figure:AlmostAlternating}. 
\begin{enumerate}
\item If neither $P$ nor $\overline{P}$ is one, then $D$ has the fewest number of crossings among all almost alternating diagrams of $L$.
\item If  $P=1$, $\overline{P}\neq 1$ and
$$P_2 - P_0+ Q - S \neq 0,$$
then the fewest number of crossings among all almost alternating diagrams of $L$ is either $c(D)$ or $c(D)-1$.
\item If $P\neq 1$, $\overline{P}=1$ and
$$\overline{P}_2 - \overline{P}_0 + \overline{Q} - \overline{S} \neq 0,$$
then the fewest number of crossings among all almost alternating diagrams of $L$ is either $c(D)$ or $c(D)-1$.
\item If both $P$ and $\overline{P}$ are one, then $L$ has another almost alternating diagram $D'$ with two fewer crossings than $D$.
\end{enumerate}
\end{theorem}
\begin{proof}
If neither $P$ nor $\overline{P}$ is one, then Theorem \ref{theorem:AAJones} implies $\Span V_L(t) = c(D)-3$, and thus $D$ has the fewest number of crossings among all almost alternating diagrams of $L$. In either cases (2) or (3), Theorem \ref{theorem:AAJones} implies $\Span V_L(t) = c(D)-4$ and the result follows.

Suppose both $P$ and $\overline{P}=1$. Then $D$ is the first diagram in Figure \ref{figure:Reduction} where each tangle $R_i$ is alternating. As shown in Figure \ref{figure:Reduction}, there is an isotopy of the link starting with $D$ and ending with another almost alternating diagram with two fewer crossings.
\end{proof}
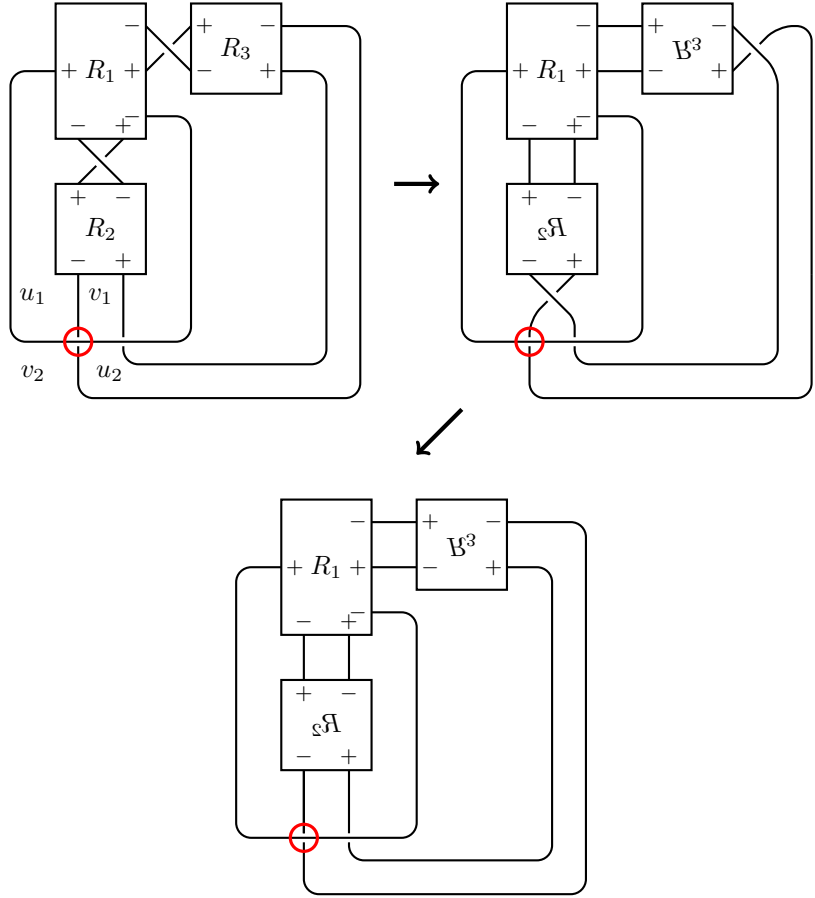
\begin{figure}[h]
$$\begin{tikzpicture}[thick, scale = .6]
\draw (0,0) rectangle (2,2);
\draw (0,3) rectangle (2,6);
\draw (3,4) rectangle (5,6);

\draw (1,4.5) node{$R_1$};
\draw (1,1) node{$R_2$};
\draw (4,5) node{$R_3$};

\draw (.3,4.5) node{\footnotesize{$+$}};
\draw (.5,3.3) node{\footnotesize{$-$}};
\draw (1.5,3.3) node{\footnotesize{$+$}};
\draw (1.7,3.5) node{\footnotesize{$-$}};
\draw (1.7,4.5) node{\footnotesize{$+$}};
\draw (1.7,5.5) node{\footnotesize{$-$}};

\draw (.5,1.7) node{\footnotesize{$+$}};
\draw (1.5,1.7) node{\footnotesize{$-$}};
\draw (.5,.3) node{\footnotesize{$-$}};
\draw (1.5,.3) node{\footnotesize{$+$}};

\draw (3.3,4.5) node{\footnotesize{$-$}};
\draw (3.3,5.5) node{\footnotesize{$+$}};
\draw (4.7,4.5) node{\footnotesize{$+$}};
\draw (4.7,5.5) node{\footnotesize{$-$}};

\draw (.5,3) -- (1.5,2);
\draw (.5,2) -- (.9,2.4);
\draw (1.1,2.6) -- (1.5,3);
\draw (2,5.5) -- (3,4.5);
\draw (2,4.5) -- (2.4,4.9);
\draw (2.6,5.1) -- (3,5.5);
\draw[rounded corners = 2mm] (2,3.5) -- (3,3.5) -- (3,-1.5) -- (-1,-1.5) -- (-1,4.5) -- (0,4.5);
\draw (1.5,0) -- (1.5,-1.4);
\draw[rounded corners = 2mm] (1.5,-1.6) -- (1.5,-2) -- (6,-2) -- (6,4.5) -- (5,4.5);
\draw (.5,0) -- (.5,-1.4);
\draw[rounded corners = 2mm] (.5,-1.6) -- (.5,-2.75) -- (6.75,-2.75) -- (6.75,5.5) -- (5,5.5);

\draw (-.5,-.5) node{$u_1$};
\draw (1.2,-2.2) node{$u_2$};
\draw (1,-.5) node{$v_1$};
\draw (-.5,-2.2) node{$v_2$};

\draw[red,very thick] (.5,-1.5) circle (.3cm);


\begin{scope}[xshift = 10cm]

\draw (0,0) rectangle (2,2);
\draw (0,3) rectangle (2,6);
\draw (3,4) rectangle (5,6);

\draw (1,4.5) node{$R_1$};
\draw (1,1) node{\reflectbox{$R_2$}};
\draw (4,5) node{\raisebox{\depth}{\scalebox{1}[-1]{$R_3$}}};
\draw (.3,4.5) node{\footnotesize{$+$}};
\draw (.5,3.3) node{\footnotesize{$-$}};
\draw (1.5,3.3) node{\footnotesize{$+$}};
\draw (1.7,3.5) node{\footnotesize{$-$}};
\draw (1.7,4.5) node{\footnotesize{$+$}};
\draw (1.7,5.5) node{\footnotesize{$-$}};

\draw (.5,1.7) node{\footnotesize{$+$}};
\draw (1.5,1.7) node{\footnotesize{$-$}};
\draw (.5,.3) node{\footnotesize{$-$}};
\draw (1.5,.3) node{\footnotesize{$+$}};

\draw (3.3,4.5) node{\footnotesize{$-$}};
\draw (3.3,5.5) node{\footnotesize{$+$}};
\draw (4.7,4.5) node{\footnotesize{$+$}};
\draw (4.7,5.5) node{\footnotesize{$-$}};

\draw (.5,3) -- (.5,2);
\draw (1.5,2) -- (1.5,3);

\draw (2,5.5) -- (3,5.5);
\draw (2,4.5) -- (3,4.5);
\draw[rounded corners = 2mm] (2,3.5) -- (3,3.5) -- (3,-1.5) -- (-1,-1.5) -- (-1,4.5) -- (0,4.5);
\draw (1.5,0) -- (1.1,-.4);
\draw[rounded corners = 2mm] (.9,-.6) -- (.5,-1) -- (.5,-1.4);
\draw[rounded corners = 2mm] (1.5,-1.6) -- (1.5,-2) -- (6,-2) -- (6,4.5) -- (5,5.5);
\draw[rounded corners = 1mm] (.5,0) -- (1.5,-1) -- (1.5,-1.4);
\draw[rounded corners = 2mm] (.5,-1.6) -- (.5,-2.75) -- (6.75,-2.75) -- (6.75,0);
\draw[rounded corners = 2mm] (6.75,0) --  (6.75,5.5) -- (6,5.5) -- (5.6,5.1);
\draw (5,4.5) -- (5.4,4.9);

\draw[red,very thick] (.5,-1.5) circle (.3cm);

\end{scope}


\begin{scope}[xshift = 5cm, yshift = -11cm]

\draw (0,0) rectangle (2,2);
\draw (0,3) rectangle (2,6);
\draw (3,4) rectangle (5,6);

\draw (1,4.5) node{$R_1$};
\draw (1,1) node{\reflectbox{$R_2$}};
\draw (4,5) node{\raisebox{\depth}{\scalebox{1}[-1]{$R_3$}}};
\draw (.3,4.5) node{\footnotesize{$+$}};
\draw (.5,3.3) node{\footnotesize{$-$}};
\draw (1.5,3.3) node{\footnotesize{$+$}};
\draw (1.7,3.5) node{\footnotesize{$-$}};
\draw (1.7,4.5) node{\footnotesize{$+$}};
\draw (1.7,5.5) node{\footnotesize{$-$}};

\draw (.5,1.7) node{\footnotesize{$+$}};
\draw (1.5,1.7) node{\footnotesize{$-$}};
\draw (.5,.3) node{\footnotesize{$-$}};
\draw (1.5,.3) node{\footnotesize{$+$}};

\draw (3.3,4.5) node{\footnotesize{$-$}};
\draw (3.3,5.5) node{\footnotesize{$+$}};
\draw (4.7,4.5) node{\footnotesize{$+$}};
\draw (4.7,5.5) node{\footnotesize{$-$}};

\draw (.5,3) -- (.5,2);
\draw (1.5,2) -- (1.5,3);

\draw (2,5.5) -- (3,5.5);
\draw (2,4.5) -- (3,4.5);
\draw[rounded corners = 2mm] (2,3.5) -- (3,3.5) -- (3,-1.5) -- (-1,-1.5) -- (-1,4.5) -- (0,4.5);
\draw (1.5,0) -- (1.5,-1.4);
\draw[rounded corners = 1mm] (.5,0) -- (.5,-1.4);
\draw[rounded corners = 2mm] (1.5,-1.6) -- (1.5,-2) -- (6,-2) -- (6,4.5) -- (5,4.5);
\draw (.5,0) -- (.5,-1.4);
\draw[rounded corners = 2mm] (.5,-1.6) -- (.5,-2.75) -- (6.75,-2.75) -- (6.75,5.5) -- (5,5.5);

\draw[red,very thick] (.5,-1.5) circle (.3cm);

\end{scope}

\draw[ultra thick, ->] (7.5,2) -- (8.5,2);
\draw[ultra thick, ->] (9,-3) -- (8,-4);

\end{tikzpicture}$$
\caption{If $P=\overline{P}=1$, then the diagram $D$ has the form of the diagram on the top left, where $R_1$, $R_2$, and $R_3$ are alternating tangles. Two flypes lead to the second diagram. A Reidemeister 3 move followed by a Reidemeister 2 move yields the third diagram. The third diagram is almost alternating and has two fewer crossings than $D$. The encircled crossing is the dealternator.}
\label{figure:Reduction}
\end{figure}

The following lemma will help us to compare the terms in Theorem \ref{theorem:AAJones}.
\begin{lemma}
\label{lemma:Dual}
Let $D$ be a strongly reduced almost alternating diagram with the fewest number of crossings among all almost alternating diagrams of the link. The following statements hold.
\begin{enumerate}
\item At least one of $P$ or $\overline{P}$ is contained in $\{0,2\}$.
\item $P+Q\leq \beta_1$ and $\overline{P}+\overline{Q} \leq \overline{\beta}_1$.
\end{enumerate} 
\end{lemma}
\begin{proof}
Suppose that $P>2$, i.e. there are more than two paths of length two between $u_1$ and $u_2$ in $G'$. Each path from $u_1$ to $u_2$ in $G'$ corresponds to possibly more than one path from $u_1$ to $u_2$ in $G$. Each edge in a path from $u_1$ to $u_2$ in $G$ is dual to an edge in $\overline{G}$ which then corresponds to an edge in $\overline{G}'$. Therefore, the shortest path between $v_1$ and $v_2$ is at least length $P$, and hence $\overline{P}=0$. Likewise, if $\overline{P}>2$, it follows that $P=0$. Since $D$ has the fewest number of crossings among all almost alternating diagrams of the link, Theorem \ref{theorem:Crossing} implies that not both $P$ and $\overline{P}$ can be one. Hence either $P$ or $\overline{P}$ is contained in $\{0,2\}$.

The edge in $G'$ corresponding to the dealternator is incident to $u_1$ and $u_2$. Each path between $u_1$ and $u_2$ increases the circuit rank $\beta_1$ by one. Since the total number of paths between $u_1$ and $u_2$ in $G'$ is at least $P+Q$, it follows that $P+Q\leq \beta_1$. The argument for the inequality $\overline{P} + \overline{Q} \leq \overline{\beta}_1$ is similar.
\end{proof}

Turaev genus one links are closely related to almost alternating links. A link is of Turaev genus one if and only if it is non-alternating and has a diagram of Turaev genus one. An almost alternating link is always Turaev genus one, but it is an open question whether there is a Turaev genus one link that is not almost alternating; see \cite{Lowrance:AltDist} for more discussion. For a more comprehensive review of the Turaev genus of a link, see Champanerkar and Kofman's recent survey \cite{CK:Survey}.

Let $L$ be a link, and let $B$ be a ball whose boundary sphere transversely intersects $L$ in four points. The pair $(L\cap B,B)$ forms a two-tangle. After an isotopy of $L$, it can be assumed that the boundary of $B$ is a round sphere and that the intersection points $L\cap\partial B$ are permuted by any $180^{\circ}$ rotation of $B$ about a coordinate axis.  A {\em mutation} of $L$ is the link obtained by removing the ball $B$ from $S^3$, rotating it $180^{\circ}$ about a coordinate axis, and gluing it back into $S^3$. Any link that can be obtained from $L$ via a sequence of mutations is a {\em mutant} of $L$. Armond and Lowrance \cite{ArmLow:Turaev} (see also \cite{Kim:TuraevClassification}) classified links of Turaev genus one and used these classifications to prove the following theorem.
\begin{theorem}[Armond, Lowrance]
\label{theorem:mutant}
Every link of Turaev genus one is mutant to an almost alternating link.
\end{theorem}

Theorem \ref{theorem:Sign} now follows from Theorem \ref{theorem:AAJones} and Lemma \ref{lemma:Dual}.
\begin{proof}[Proof of Theorem \ref{theorem:Sign}]
Suppose that $D$ is a strongly reduced almost alternating diagram of $L$. By Lemma \ref{lemma:Dual} either $P$ or $\overline{P}$ is contained in $\{0,2\}$. Without loss of generality, suppose that $P\in\{0,2\}$. First, let $P=0$. Then $P_0=P_2=S=0$,
\begin{align*}
\alpha_0 = & \; (-1)^{v+1},~\text{and}\\
\alpha_1 = & \; (-1)^{v-1}(Q-\beta_1).
\end{align*}
By Lemma \ref{lemma:Dual}, we have that $Q-\beta_1 \leq 0$, and thus $\alpha_0 \alpha_1 \leq 0$. 

Now suppose that $P=2$. Then 
\begin{align*}
\alpha_0 = & \; (-1)^v~\text{and}\\
\alpha_1 = & \; (-1)^{v-1} (\beta_1 - 1 + P_2-P_0+Q - S).
\end{align*}
Since $P=2$, it follows that $-2\leq P_2 - P_0 \leq 2$, $S\in \{0,1\}$, and $\beta_1\geq 2$. Suppose $S=0$. The quantity $\beta_1-1+P_2-P_0+Q-S$ achieves its minimum of $-1$ when $\beta_1=2, P_2 =0, P_0=2,$ and $Q=0$. In this case, $G$ is a $4$-cycle with an additional edge between $u_1$ and $u_2$, and $D$ is a diagram of the two component unlink, as shown in Figure \ref{figure:S=0}. Hence if $D$ is a diagram of an almost alternating link and $S=0$, then $\beta_1-1+P_2-P_0+Q-S\geq 0$.
\begin{figure}[h]
$$\begin{tikzpicture}


\draw (0.2,0) node[below]{\small{$u_2$}};
\draw (0.2,2) node[above]{\small{$u_1$}};
\draw (0,-1) node[below]{$G$};

\fill (0,0) circle (.1cm);
\fill (1,1) circle (.1cm);
\fill (-1,1) circle (.1cm);
\fill (0,2) circle (.1cm);
\draw (0,0) -- (1,1) -- (0,2) -- (-1,1) -- (0,0);
\draw[dashed] (0,2) arc (0:180:.75cm);
\draw[dashed]  (0,0) arc (0:-180:.75cm);
\draw[dashed]  (-1.5,0) -- (-1.5,2);

\begin{scope}[xshift = 4cm, thick]

\draw (0,.6) -- (0,2.1);
\draw (0,2.1) arc (0:180:.4cm and .3cm);
\draw (-.1,1.5) arc (90:270:.4cm and .5cm);
\draw (-.1,.5) -- (.9,.5);
\draw (0,.4) -- (0,-.1);
\draw (0,-.1) arc (0:-180:.4cm and .3cm);
\draw (1,1.4) -- (1,-.3);
\draw (.1,1.5) -- (1.1,1.5);
\draw (1.1,.5) arc (-90:90:.4cm and .5cm);
\draw[rounded corners = 2mm] (-.8,2.1) -- (-.8,1.5) -- (-1.8,.5) -- (-1.8,-.3);
\draw[rounded corners = 2mm] (-.8,-.1) -- (-.8,.5) -- (-1.2, .9);
\draw[rounded corners = 2mm] (-1.4,1.1) -- (-1.8,1.5) -- (-1.8,2.3);
\draw (-1.8,2.3) arc (180:90:.4cm);
\draw (-1.8,-.3) arc (180:270:.4cm);
\draw (-1.4,-.7) -- (.6,-.7);
\draw (.6,-.7) arc (-90:0:.4cm);
\draw (-1.4,2.7) -- (.6,2.7);
\draw (.6,2.7) arc (90:0:.4cm);
\draw (1,2.3) -- (1,1.6);

\draw[ultra thick,->] (2,1) -- (3,1);

\draw (4,1) circle (.6cm);
\draw (5.5,1) circle (.6cm);
\draw (4.75,-1) node[below]{$\bigcirc \sqcup \bigcirc$};

\draw (-.5,-1) node[below]{$D$};

\end{scope}

\end{tikzpicture}$$
\caption{If $S=0$ and $\alpha_1= (-1)^v$, then $G$ and $D$ are as above, and $D$ is a diagram of the two component unlink. Hence the link is alternating, rather than almost alternating.}
\label{figure:S=0}
\end{figure}
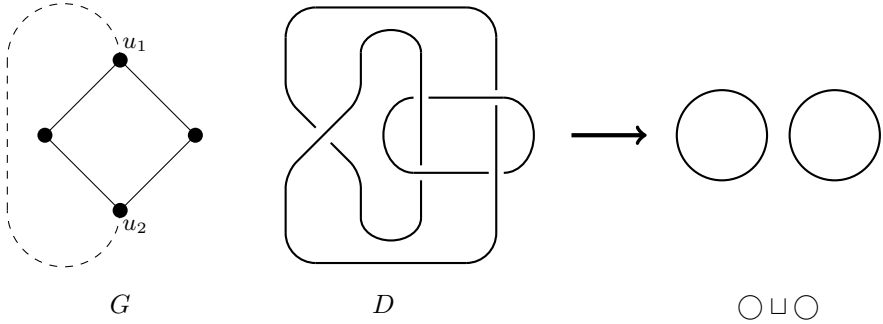

Now suppose that $S=1$. Then $\beta_1 \geq 3$. The quantity $\beta_1 -1 + P_2 - P_0 + Q -S$ achieves its minimum of $-1$ when $\beta_1=3, P_2=0, P_0=2$ and $Q=0$. In this case, $G'$ is $K_4$ and $G$ is the graph in Figure \ref{figure:S=1}. The diagram $D$ is a diagram of the $(2,k+1)$-torus link disjoint union with an unknot. Hence if $D$ is a diagram of an almost alternating link and $S=1$, then $\beta_1-1+P_2-P_0+Q-S\geq 0$.

\begin{figure}[h]
$$\begin{tikzpicture}

\draw (0.2,0) node[below]{\small{$u_2$}};
\draw (0.2,2) node[above]{\small{$u_1$}};
\draw (0,-1) node[below]{$G$};

\fill (0,0) circle (.1cm);
\fill (1,1) circle (.1cm);
\fill (-1,1) circle (.1cm);
\fill (0,2) circle (.1cm);
\draw (0,0) -- (1,1) -- (0,2) -- (-1,1) -- (0,0);
\draw (-1,1) -- (1,1);
\draw (0,1) node[above]{$k$};
\draw[dashed] (0,2) arc (0:180:.75cm);
\draw[dashed]  (0,0) arc (0:-180:.75cm);
\draw[dashed]  (-1.5,0) -- (-1.5,2);

\begin{scope}[xshift = 4cm, thick]

\draw (0,2) -- (0,2.1);
\draw (0,2.1) arc (0:180:.4cm and .3cm);
\draw (-.1,1.9) arc (90:270:.5cm and .9cm);
\draw (-.1,.1) -- (.9,.1);
\draw (0,0) -- (0,-.1);
\draw (0,-.1) arc (0:-180:.4cm and .3cm);
\draw (1,.2) -- (1,-.3);
\draw (1,.2) arc (0:90:.2cm);
\draw (.8,.4) arc (-90:-180:.2cm);
\draw (.1,1.9) -- (1.1,1.9);
\draw (0,.2) arc (180:90:.2cm);
\draw (.2,.4) arc (-90:0:.2cm);
\draw (1,1.8) arc (0:-90:.2cm);
\draw (.8,1.6) arc (90:180:.2cm);
\draw[rounded corners = 1mm] (.6,1.4) -- (.6,1.3) -- (.4,1.1);
\draw[rounded corners = 1mm] (.4,1.4) -- (.4,1.3) -- (.45,1.25);
\draw (.55,1.15) -- (.6,1.1);

\draw[rounded corners = 1mm] (.4,.6) -- (.4,.7) -- (.6,.9);
\draw[rounded corners = 1mm] (.6,.6) -- (.6,.7) -- (.55,.75);
\draw (.45,.85) -- (.4,.9);

\draw[decoration={brace,amplitude=2pt}, decorate] (.3,.6) -- (.3,1.4);
\draw (.3,1) node[left]{$k$};

\fill (.5,1) circle (.03cm);
\fill (.5,1.1) circle (.03cm);
\fill (.5,.9) circle (.03cm);

\draw (0,2.1) -- (0,1.8);
\draw (0,1.8) arc (180:270:.2cm);
\draw (.2,1.6) arc (90:0:.2cm);
\draw (1.1,.1) arc (-90:90:.5cm and .9cm);
\draw[rounded corners = 2mm] (-.8,2.1) -- (-.8,1.5) -- (-1.8,.5) -- (-1.8,-.3);
\draw[rounded corners = 2mm] (-.8,-.1) -- (-.8,.5) -- (-1.2, .9);
\draw[rounded corners = 2mm] (-1.4,1.1) -- (-1.8,1.5) -- (-1.8,2.3);
\draw (-1.8,2.3) arc (180:90:.4cm);
\draw (-1.8,-.3) arc (180:270:.4cm);
\draw (-1.4,-.7) -- (.6,-.7);
\draw (.6,-.7) arc (-90:0:.4cm);
\draw (-1.4,2.7) -- (.6,2.7);
\draw (.6,2.7) arc (90:0:.4cm);
\draw (1,2.3) -- (1,2);

\draw (-.5,-1) node[below]{$D$};

\end{scope}

\draw[ultra thick, ->] (5.8,1) -- (6.7,1);

\begin{scope}[xshift = 7.5 cm, thick]

\draw[rounded corners = 1mm] (0,0) -- (0,.2) -- (.5,.7);
\draw[rounded corners = 1mm] (.5,0) -- (.5,.2) -- (.3,.4);
\draw[rounded corners = 1mm] (.2,.5) -- (0,.7);
\draw[rounded corners = 1mm] (0,2) -- (0,1.8) -- (.5,1.3);
\draw[rounded corners = 1mm] (.5,2) -- (.5,1.8) -- (.3,1.6);
\draw (.2,1.5) -- (0,1.3);
\draw (0,0) arc (0:-180:.2cm);
\draw (0,2) arc (0:180:.2cm);
\draw (-.4,0) -- (-.4,2);
\draw (.5,2) arc (180:0:.2cm);
\draw (.5,0) arc (180:360:.2cm);
\draw (.9,0) -- (.9,2);

\fill (0.25,1) circle (.03cm);
\fill (0.25,1.1) circle (.03cm);
\fill (0.25,.9) circle (.03cm);

\draw (2,1) circle (.6cm);

\draw (1,-1) node[below]{$T_{2,k+1}\sqcup \bigcirc$};

\end{scope}

\end{tikzpicture}$$
\caption{If $S=1$ and $\alpha_1=(-1)^{v}$, then $G$ and $D$ are as above, and $D$ is a diagram of the disjoint union of the $(2,k+1)$ torus knot $T_{2,k+1}$ and the unknot. Hence the link is alternating, rather than almost alternating.}
\label{figure:S=1}
\end{figure}
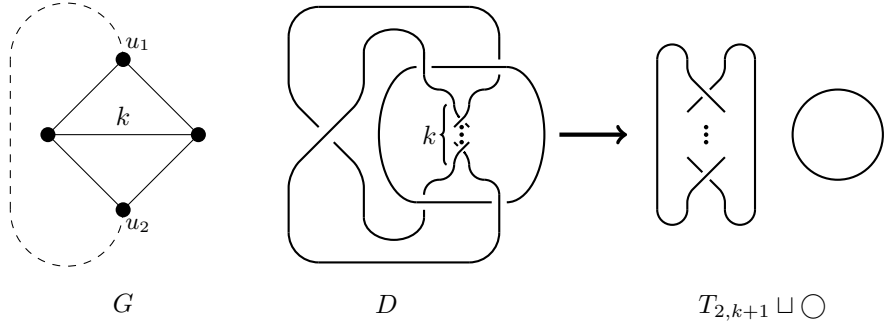

In both cases where $P=2$ and $S=0$ or $S=1$, we have
\begin{align*}
\alpha_0 \alpha_1 = & \; (-1)^{v}(-1)^{v-1}(\beta_1-1+P_2-P_0+Q-S)\\
=& \; (-1)(\beta_1-1+P_2-P_0+Q-S)\\
 \leq & \; (-1)(0)\\
 \leq & \; 0.
\end{align*}
If $\overline{P}\in \{0,2\}$, then a symmetric argument implies $\alpha_{c-4}\alpha_{c-3}\leq 0$. If $L$ is Turaev genus one, then Theorem \ref{theorem:mutant} implies that $L$ is mutant to an almost alternating link. It is a well-known fact that mutant links have the same Jones polynomial; see, for example, \cite{Lickorish:Book}. Hence, the result follows.
\end{proof}

Let $V_K(t) =a_0 t^{k} + a_1 t^{k+1} + \cdots + a_{n-1} t^{k+n-1}+a_n t^{k+n}$ for a knot $K$. Dasbach and Lowrance \cite{DasLow:TuraevJones} proved that if $K$ is almost alternating, then at least one of $|a_0|$ or $|a_n|$ is one, and Theorem \ref{theorem:Sign} gives a stronger obstruction for a knot to be almost alternating.  A computer search (using the tables from \cite{knotinfo} and the Jones polynomial program from \cite{knotatlas}) shows 1 knot with eleven crossings, 11 knots with twelve crossings, 70 knots with thirteen crossings, 526 knots with fourteen crossings, and 3,787 knots with fifteen crossings have Jones polynomials where neither the leading nor trailing coefficients have absolute value one. However, knots whose Jones polynomials have either a leading or trailing coefficient of absolute value one but fail to satisfy the conditions of Theorem \ref{theorem:Sign} are not as common in the knot tables. There are no such knots with fourteen or fewer crossings and 15 such knots with fifteen crossings. Example \ref{example:15n41133} shows one of the fifteen crossing examples.
\begin{example}
\label{example:15n41133}
The knot $K=$15n41133 of Figure \ref{figure:15n41133} has Jones polynomial 
$$V_K(t) = t^4+t^5-3t^6+8t^7-12t^8+14t^9-15t^{10}+13t^{11}-10t^{12} + 6t^{13} - 2t^{14}.$$
Thus Theorem \ref{theorem:Sign} implies that $K$ is not almost alternating and has Turaev genus at least two.
\end{example}

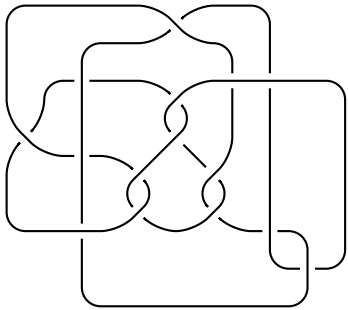
\begin{figure}[h]
$$\begin{tikzpicture}[thick, scale = .5]

\draw[rounded corners = 2.5mm] (5.7,6.3) -- (6,6) -- (4,4) -- (4.3,3.7);
\draw[rounded corners = 2.5mm] (5.3,5.7) -- (5,6) -- (6,7)  -- (10,7) -- (10,2) -- (9.2,2);
\draw (5.7,5.3) -- (6.3,4.7);
\draw[rounded corners = 2.5mm] (6.7,4.3) -- (7,4) -- (6,3) -- (5,3) -- (4.7,3.3);
\draw[rounded corners = 2.5mm] (4.7,4.3) -- (5,4) -- (4,3) -- (1,3) -- (1,5) -- (1.3,5.3);
\draw[rounded corners = 2.5mm] (4.3,4.7) -- (4,5) -- (3.2,5);
\draw[rounded corners = 2.5mm] (6.3,3.7) -- (6,4) -- (7,5) -- (7,6.8);
\draw[rounded corners = 2.5mm] (5.3,6.7) -- (5,7) -- (3.2,7);
\draw[rounded corners = 2.5mm] (3,3.2) -- (3,8) -- (5,8) -- (5.3,8.3);
\draw[rounded corners = 2.5mm] (2.8,5) -- (2,5) -- (1,6) -- (1,9) -- (5,9) -- (6,8) -- (7,8) -- (7,7.2);
\draw[rounded corners = 2.5mm] (5.7,8.7) -- (6,9) -- (8,9) -- (8,7.2);
\draw[rounded corners = 2.5mm] (2.8,7) -- (2,7) -- (2,6) -- (1.7,5.7);
\draw[rounded corners = 2.5mm] (8,6.8) -- (8,2) -- (8.8,2);
\draw[rounded corners = 2.5mm] (6.7,3.3) -- (7,3) -- (7.8,3);
\draw[rounded corners = 2.5mm] (8.2,3) -- (9,3) -- (9,1) -- (3,1) -- (3,2.8);

\end{tikzpicture}$$
\label{figure:15n41133}
\caption{The knot 15n41133.}
\end{figure}

\section{Non-triviality of the Jones polynomial}
\label{section:JonesUnknot}

In this section we prove that the Jones polynomial of an almost alternating or Turaev genus one link is not equal to any unit times the Jones polynomial of an unlink. Theorems \ref{theorem:Sign} and \ref{theorem:AAJones} are the main tools used in the proof, but there are some exceptional cases not covered by those theorems. 

Throughout this section, $D$ is a strongly reduced almost alternating diagram with the fewest number of crossings among all almost alternating diagrams of the link. Furthermore, we assume that $D$ is a prime diagram, i.e. there is no simple closed curve $\gamma$ meeting $D$ exactly twice away from the crossings such that both the interior and exterior of $\gamma$ contain crossings. If $D$ is not prime, then one of its factors is alternating. Since the Jones polynomial is multiplicative under connected sum and the Jones polynomial of a nontrivial alternating link is nontrivial, it suffices to consider prime almost alternating diagrams. Since $D$ is a prime diagram, it follows that $G$ and $\overline{G}$ are two-connected, i.e. contain no cut vertices. We adopt the notation of Theorem \ref{theorem:AAJones} and also use the diagrammatic notation of Figure \ref{figure:Shorthand}.
\begin{figure}[h]
$$\begin{tikzpicture}[thick]

\fill (0,0) circle (.1cm);
\fill (1,0) circle (.1cm);
\draw (.5,0) node[above]{$a$};
\draw (0,0) -- (1,0);
\draw (.5,-2) node{in $G$ or $\overline{G}$};


\draw (3,-1) rectangle (4,1);
\draw (3.5,0) node{$a$};
\draw (3.25,1) -- (3.25,1.25);
\draw (3.75,1) -- (3.75,1.25);
\draw (3.25,-1) -- (3.25,-1.25);
\draw (3.75,-1) -- (3.75,-1.25);
\draw (3.25,.8) node{$-$};
\draw (3.75,.8) node{$+$};
\draw (3.25,-.8) node{$+$};
\draw (3.75,-.8) node{$-$};
\draw (3.5,-2) node{schematic in $D$};


\draw[rounded corners = 1mm] (6.25,-1.25) --  (6.25,-1) -- (6.75, -.5);
\draw[rounded corners = 1mm] (6.75,-1.25) --  (6.75, -1) -- (6.6,-.85);
\draw (6.4,-.65) -- (6.25,-.5);

\draw[decoration={brace,amplitude=5pt}, decorate] (5.75,-1) -- (5.75,1);
\draw (5.6,0) node[left]{$a$};

\begin{scope}[yshift = 1.5cm]
\draw[rounded corners = 1mm] (6.25,-1) -- (6.75, -.5) -- (6.75,-.25);
\draw[rounded corners = 1mm] (6.75, -1) -- (6.6,-.85);
\draw[rounded corners = 1mm] (6.4,-.65) -- (6.25,-.5) -- (6.25,-.25);
\end{scope}

\fill (6.5,0) circle (.05cm);
\fill (6.5,.2) circle (.05cm);
\fill (6.5,-.2) circle (.05cm);

\draw (6.5,-2) node{in $D$};

\begin{scope}[yshift = -4cm, xshift = 1.5cm]

\begin{scope}[xshift = -3cm]
	\draw (0.5,0) circle (.5cm);
	\draw (.5,0) node{$\vec{a}$};
	\draw (.5,.5) -- (.5,1);
	\draw (.5,-.5) -- (.5,-1);
	\fill (.5,1) circle (.1cm);
	\fill (.5,-1) circle (-.1cm);
	
	\draw (.5,-2) node{schematic in};
	\draw (.5,-2.4) node{$G$ or $\overline{G}$};
	
\end{scope}
	\draw (0,.7) -- (.5,0) -- (1,.7);
	\draw (0,-.7) -- (.5,0) -- (1,-.7);
	\fill[white] (0.5,0) circle (.5cm);
	\draw (0.5,0) circle (.5cm);
	\draw (.5,0) node{$\vec{a}$};
	\draw (.3,.3) node{\small{$-$}};
	\draw (.7,.3) node{\small{$+$}};
	\draw (.3,-.3) node{\small{$+$}};
	\draw (.7,-.3) node{\small{$-$}};
	
	\draw (.5,-2) node{schematic in $D$};

	\fill (3.5,1) circle (.1cm);
	\fill (3.5,.4) circle (.1cm);
	\fill (3.5,-.4) circle (.1cm);
	\fill (3.5,-1) circle (.1cm);
	\draw (3.5,1) -- (3.5,.4);
	\draw (3.5,-.4) -- (3.5,-1);
	\fill (3.5,0) circle (.05cm);
	\fill (3.5,-.2) circle (.05cm);
	\fill (3.5,.2) circle (.05cm);
	\draw (3.5,.7) node[left]{$a_1$};
	\draw (3.5,-.7) node[left]{$a_k$};
	
	\draw (3.5,-2) node{in $G$ or $\overline{G}$};
	

	\draw[rounded corners = 1mm] (6.75,1.25) -- (6.75,1) -- (7,1) -- (7,.8) -- (6.7,.5);
	\draw[rounded corners = 1mm] (7,.4) --  (7,.5) -- (6.9,.6);
	\draw (6.8,.7) -- (6.7,.8);
	\draw[rounded corners = 1mm] (6.25,1.25) -- (6.25,1) -- (6,1) -- (6,.8) -- (6.1,.7);
	\draw (6.2,.6) -- (6.3,.5);
	\draw[rounded corners = 1mm] (6,.4) --  (6,.5) -- (6.3,.8);
	\fill (6.5,.65) circle (.03cm);
	\fill (6.4,.65) circle (.03cm);
	\fill (6.6, .65) circle (.03cm);
	
	\draw[decoration={brace,amplitude=2pt}, decorate] (6.2,.85) -- (6.8,.85);
	\draw (6.5,1.1) node{\small{$a_1$}};
	
	\begin{scope}[yscale = -1, xscale = -1, xshift = -13cm]
	
	\draw[rounded corners = 1mm] (6.75,1.25) -- (6.75,1) -- (7,1) -- (7,.8) -- (6.7,.5);
	\draw[rounded corners = 1mm] (7,.4) --  (7,.5) -- (6.9,.6);
	\draw (6.8,.7) -- (6.7,.8);
	\draw[rounded corners = 1mm] (6.25,1.25) -- (6.25,1) -- (6,1) -- (6,.8) -- (6.1,.7);
	\draw (6.2,.6) -- (6.3,.5);
	\draw[rounded corners = 1mm] (6,.4) --  (6,.5) -- (6.3,.8);
	\fill (6.5,.65) circle (.03cm);
	\fill (6.4,.65) circle (.03cm);
	\fill (6.6, .65) circle (.03cm);
	\end{scope}
	
	\begin{scope}[yscale = -1]
	\draw[decoration={brace,amplitude=2pt,mirror}, decorate] (6.2,.85) -- (6.8,.85);
	\draw (6.5,1.1) node{\small{$a_k$}};
	\end{scope}

	\fill (6.5,0) circle (.05cm);
	\fill (6.5,.2) circle (.05cm);
	\fill (6.5,-.2) circle (.05cm);
	\draw[decoration={brace,amplitude=5pt}, decorate] (5.75,-1) -- (5.75,1);
	\draw (5.6,0) node[left]{$k$};
 
	\draw (6.5,-2) node{in $D$};

\end{scope}

\end{tikzpicture}$$
\caption{\textbf{Top.} A multiple edge in $G$ corresponds to a twist region in $D$. We use the shorthand of a box labeled with the number of crossings in the twist region. \textbf{Bottom.} A circle labeled with $\vec{a} = (a_1,\dots, a_k)$ corresponds to a series of $k$ edges in $G'$ or $\overline{G}'$ and to the depicted alternating tangle in $D$.}
\label{figure:Shorthand}
\end{figure}
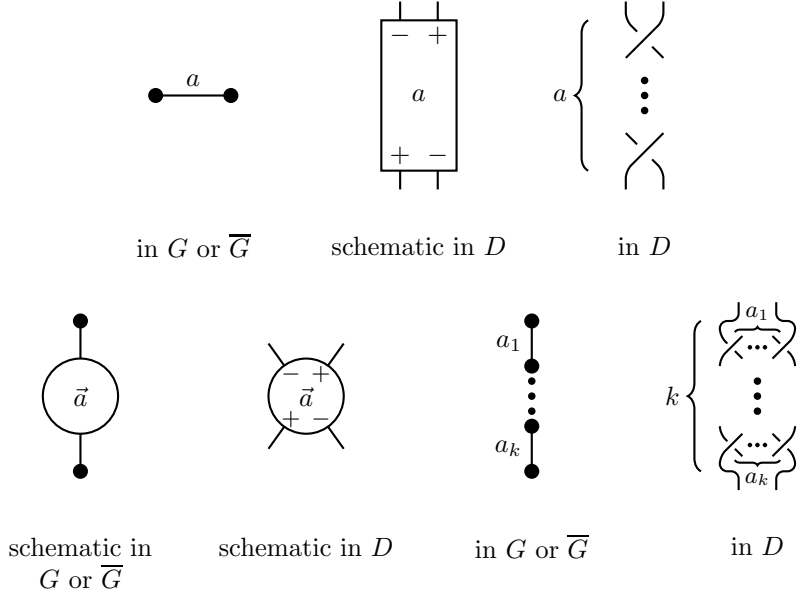

Let $\Gamma$ be a connected graph, and suppose that $\Gamma$ has two vertices $w_1$ and $w_2$ such that $\Gamma - \{w_1,w_2\}$ is disconnected. Then $\Gamma$ can be expressed as the union of two connected subgraphs $\Gamma=\Gamma_1\cup\Gamma_2$ where $\Gamma_1\cap\Gamma_2=\{w_1,w_2\}$. Consider $\Gamma_1$ and $\Gamma_2$ as distinct graphs. Temporarily let $w^1_1$ and $w^1_2$ be the copies of $w_1$ and $w_2$ in $\Gamma_1$, and let $w^2_1$ and $w^2_2$ be the copies of $w_1$ and $w_2$ in $\Gamma_2$. Define $\Gamma'$ be the graph obtained by gluing together $w^1_1$ and $w^2_2$ and by gluing together $w^2_1$ and $w^1_2$. The operation described above is called a {\em $2$-isomorphism}, and any two graphs related by a sequence of $2$-isomorphisms are said to be {\em $2$-isomorphic}. 

A $2$-isomorphism on the checkerboard graph of a link diagram corresponds to a mutation of the link, and mutation does not affect the Jones polynomial. In our setting, we will use $2$-isomorphisms to permute the labels along a path in $G$ or $\overline{G}$, which will decrease the number of cases we need to consider in Lemmas \ref{lemma:F1-3} and \ref{lemma:F4-7} below.

A link diagram is {\em $A$-adequate} (respectively {\em $B$-adequate}) if no trace in its all-$A$ (respectively all-$B$) Kauffman state has both of its endpoints on the same component of the Kauffman state. A link that has a diagram that is either $A$-adequate or $B$-adequate is called {\em semi-adequate}. Stoimenow \cite{Stoimenow:Semi} proved the following theorem about the Jones polynomial of a semi-adequate link.
\begin{theorem}[Stoimenow]
\label{theorem:Semi-adequate}
Let $L$ be a semi-adequate link. If 
$$V_L(t) = t^k \left(-t^{\frac{1}{2}}-t^{-\frac{1}{2}}\right)^{\ell-1}$$
for some $k\in\mathbb{Z}$, then $k=0$ and $L$ is the $\ell$-component unlink.
\end{theorem}

Before proving Theorem \ref{theorem:AANontrivial}, we need Lemmas \ref{lemma:F1-3} and \ref{lemma:F4-7}.
\begin{lemma}
\label{lemma:F1-3}
Let $D$ be a strongly reduced almost alternating diagram. Suppose that $\alpha_1=0$, $\alpha_{c-4}=0$, $P=1$, and $\overline{P}=0$. After possibly relabeling $v_1$ and $v_2$, the checkerboard graph $\overline{G}$ is $2$-isomorphic to one of the three families in Figure \ref{figure:F1-3}. Moreover, every link whose checkerboard graph is in any of these families is semi-adequate.
\end{lemma}

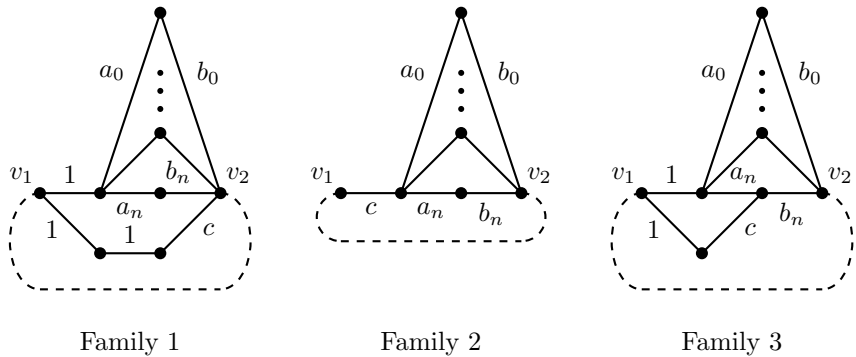
\begin{figure}[h]
$$\begin{tikzpicture}[thick, scale = .8]

\begin{scope}[xshift = 5cm]

	\fill (0,0) circle (.1cm);
	\fill (1,0) circle (.1cm);
	\fill (2,0) circle (.1cm);
	\fill (3,0) circle (.1cm);
	\fill (2,1) circle (.1cm);
	\fill (2,3) circle (.1cm);
	\draw (-.3,.3) node{$v_1$};
	\draw (3.3,0.3) node{$v_2$};
	
	\fill (2,1.4) circle (.05cm);
	\fill (2,1.7) circle (.05cm);
	\fill (2,2) circle (.05cm);
	
	\draw (0,0) -- (3,0) -- (2,3) -- (1,0);
	\draw (1,0) -- (2,1) -- (3,0);
	\draw[dashed] (0,0) arc (90:270:.4cm);
	\draw[dashed] (0,-.8) -- (3,-.8);
	\draw[dashed] (3,0) arc (90:-90:.4cm);
	\draw (.5,0) node[below]{$c$};
	\draw (1.5,0) node[below]{$a_n$};
	\draw (2.5,0) node[below]{$b_n$};
	\draw (1.2,2) node{$a_0$};
	\draw (2.8,2) node{$b_0$};

	\draw (1.5,-2.5) node{Family 2};
	
\end{scope}
	

	\fill (0,0) circle (.1cm);
	\fill (1,0) circle (.1cm);
	\fill (2,0) circle (.1cm);
	\fill (3,0) circle (.1cm);
	\fill (2,1) circle (.1cm);
	\fill (2,3) circle (.1cm);
	\fill (1,-1) circle (.1cm);
	\fill (2,-1) circle (.1cm);
	\draw (-.3,.3) node{$v_1$};
	\draw (3.3,0.3) node{$v_2$};
	
	\fill (2,1.4) circle (.05cm);
	\fill (2,1.7) circle (.05cm);
	\fill (2,2) circle (.05cm);

	\draw (0,0) -- (3,0) -- (2,3) -- (1,0);
	\draw (1,0) -- (2,1) -- (3,0);
	\draw (0,0) -- (1,-1) -- (2,-1) -- (3,0);
	\draw[dashed] (0,0) arc (90:270:.5cm and .8cm);
	\draw[dashed] (3,0) arc (90:-90:.5cm and .8cm);
	\draw[dashed] (0,-1.6) -- (3,-1.6);
	\draw (.5,0) node[above]{$1$};
	\draw (.2,-.6) node{$1$};
	\draw (1.5,-1)node[above]{$1$};
	\draw (2.8,-.6) node{$c$};
	\draw (1.5,0) node[below]{$a_n$};
	\draw (2.3,0) node[above]{$b_n$};
	\draw (1.2,2) node{$a_0$};
	\draw (2.8,2) node{$b_0$};
	
	\draw (1.5,-2.5) node{Family 1};

\begin{scope}[xshift = 10cm]

	\fill (0,0) circle (.1cm);
	\fill (1,0) circle (.1cm);
	\fill (2,0) circle (.1cm);
	\fill (3,0) circle (.1cm);
	\fill (2,1) circle (.1cm);
	\fill (2,3) circle (.1cm);
	\fill (1,-1) circle (.1cm);
	\draw (-.3,.3) node{$v_1$};
	\draw (3.3,.3) node{$v_2$};
	
	\fill (2,1.4) circle (.05cm);
	\fill (2,1.7) circle (.05cm);
	\fill (2,2) circle (.05cm);

	\draw (0,0) -- (3,0) -- (2,3) -- (1,0);
	\draw (1,0) -- (2,1) -- (3,0);
	\draw (0,0) -- (1,-1) -- (2,0);
	\draw[dashed] (0,0) arc (90:270:.5cm and .8cm);
	\draw[dashed] (3,0) arc (90:-90:.5cm and .8cm);
	\draw[dashed] (0,-1.6) -- (3,-1.6);
	\draw (.5,0) node[above]{$1$};
	\draw (.2,-.6) node{$1$};
	\draw (1.7,0) node[above]{$a_n$};
	\draw (2.5,0) node[below]{$b_n$};
	\draw (1.2,2) node{$a_0$};
	\draw (2.8,2) node{$b_0$};
	\draw (1.8,-.6) node{$c$};

	\draw (1.5,-2.5) node{Family 3};

\end{scope}

\end{tikzpicture}$$
\caption{Three families of $\overline{G}$ when $\alpha_1=0$, $\alpha_{c-4}=0$, $P=1$, and $\overline{P}=0$.}
\label{figure:F1-3}
\end{figure}

\begin{proof}
We begin the proof by making observations that will apply to all three families. Since $\overline{P}=0$, there are no paths of length two between $v_1$ and $v_2$ in $\overline{G}'$, and thus $\overline{P}_2 =\overline{P}_0=\overline{S}=0$. Theorem \ref{theorem:AAJones} implies that
$$\alpha_{c-4} = (-1)^{\overline{v}-1}\left( \overline{Q}-\overline{\beta}_1\right),$$
 and because $\alpha_{c-4}=0$, it follows that
$\overline{Q} = \overline{\beta}_1$. Furthermore, since the number of paths $P$ of length two between $u_1$ and $u_2$ in $G'$ is one, it follows that there are no $K_4$ subgraphs of $G'$ containing $u_1$ and $u_2$, and hence $S=0$. Theorem \ref{theorem:AAJones} implies that $\alpha_1 =  (-1)^{v-1}(P_2- P_0+Q).$ Because $\alpha_1=0$, we have $P_0 = P_2 + Q$. Since $P_i$ is the number of paths of length two in $G'$ with $i$ edges coming from multiple edges in $G$, it follows that $P_0+P_2\leq P = 1$. Therefore either $P_0=P_2=Q=0$ or $P_0=Q=1$ and $P_2=0$. 

We construct all graphs satisfying the specified conditions from an initial graph that contains only the vertices $v_1$ and $v_2$ and a single edge incident to both vertices. The edge is associated to the dealternator. From the upcoming construction, it will be clear that all of the graphs will be $2$-isomorphic to a graph in one of Family 1, 2, or 3, where for Family 1, the equation $P_0=P_2=Q=0$ is satisfied, and for Families 2 and 3, either set of equations ($P_0=P_2=Q=0$, or $P_0=Q=1$ and $P_2=0$) can be satisfied, depending on the parameters. For the initial graph, we have $\overline{Q}=\overline{\beta}_1 = 0$. Each time a path of length three between $v_1$ and $v_2$ is added to $\overline{G}'$, both $\overline{Q}$ and $\overline{\beta}_1$ increase by one. If a path of any length is added between any two existing vertices is added to $\overline{G}'$, then $\overline{\beta}_1$ increases by one. Therefore, every graph satisfying $\alpha_{c-4}=0$ and $\overline{P}=0$ can be obtained from our initial graph by adding paths of length three between $v_1$ and $v_2$. We consider two types of path additions. A type 1 path addition adds a path between $v_1$ and $v_2$ where both interior vertices in the new path do not exist in the previous graph, and a type 2 path addition adds a path between $v_1$ and $v_2$ where exactly one interior vertex in the new path does not exist in the previous graph. 

Any graph $\overline{G}'$ where $\overline{Q}=\overline{\beta}_1$ can be obtained from our initial graph by first performing some number of type 1 additions, then performing some number of type 2 additions. If more than two type 1 additions are performed, then there cannot be a path of length two between $u_1$ and $u_2$ in the dual graph. Hence the number of type 1 additions is one or two. Suppose the number of type 1 additions is two, and let $u_3$ be the vertex in the dual graph corresponding to the face between the two type 1 paths. Since $P=1$, the vertex $u_3$ is adjacent to both $u_1$ and $u_2$. Moreover, every path between $u_1$ and $u_2$ must contain the vertex $u_3$. Hence $Q=0$, and thus $P_0=P_2=0$ while $P_1=1$. Therefore two of the edges, say $e_1$ and $e_2$, along one of the type 1 path additions must be labeled 1. A different choice for the two edges labeled $1$ yields a $2$-isomorphic graph. Moreover, no type 2 path additions can use any interior vertex incident to either $e_1$ or $e_2$. Thus all type 2 additions must be performed along the existing path between $v_1$ and $v_2$ that does not contain $e_1$ and $e_2$. In order to ensure $P=1$, all such type 2 additions must share an edge labeled 1. The resulting family of graphs is Family 1.

Now suppose there is only one type 1 path addition to obtain $\overline{G}'$. Label the interior vertices of this path $v_3$ and $v_4$ with $v_3$ adjacent to $v_1$ and $v_4$ adjacent to $v_2$. One can perform an arbitrary number of type 2 path additions using vertex $v_3$ but not vertex $v_4$, and as long as the edge incident to $v_3$ and $v_1$ is suitably labeled, the condition $P=1$ can still be satisfied. The resulting family of graphs is Family 2.

Suppose one performs a type 2 path addition using vertex $v_3$ and another type 2 path addition using vertex $v_4$. In order to keep $P=1$, all the remaining type 2 path additions must use either $v_3$ or $v_4$, but not both.  In order to ensure $P=1$, at least one edge incident to either $v_1$ or $v_2$ in the original type 1 path must be labeled 1, and at least one edge in the solitary type 2 addition must be labeled 1. As before, a different choice of edge to label 1 results in a 2-isomorphic graph. The resulting family of graphs is Family 3.

Figure \ref{figure:F123} shows that all the links whose checkerboard graphs are in the three families of Figure \ref{figure:F1-3} are $B$-adequate and hence semi-adequate. In each case, the all-$B$ state is drawn.

\end{proof}

\begin{figure}[h]
$$\begin{tikzpicture}[thick, scale = .54, yscale = .8]

\begin{scope}[yshift = -13cm]

	\draw (0,0) rectangle (1,1.5);
	\draw (.5,.75) node{\tiny{$c$}};
	\draw (.2,-.2) node[above]{\tiny{$-$}};
	\draw (.8,-.2) node[above]{\tiny{$+$}};
	\draw (.8,1.7) node[below]{\tiny{$-$}};
	\draw (.2,1.7) node[below]{\tiny{$+$}};
	
	\begin{scope}[xshift = 3cm]
	\draw (0,0) rectangle (1,1.5);
	\draw (.5,.75) node{\tiny{$a_n$}};
	\draw (.2,-.2) node[above]{\tiny{$-$}};
	\draw (.8,-.2) node[above]{\tiny{$+$}};
	\draw (.8,1.7) node[below]{\tiny{$-$}};
	\draw (.2,1.7) node[below]{\tiny{$+$}};
	\end{scope}
	
	\begin{scope}[xshift = 5cm]
	\draw (0,0) rectangle (1,1.5);
	\draw (.5,.75) node{\tiny{$b_n$}};
	\draw (.2,-.2) node[above]{\tiny{$-$}};
	\draw (.8,-.2) node[above]{\tiny{$+$}};
	\draw (.8,1.7) node[below]{\tiny{$-$}};
	\draw (.2,1.7) node[below]{\tiny{$+$}};
	\end{scope}
	
	\begin{scope}[xshift = 3cm, yshift = 4.5cm]
	\draw (0,0) rectangle (1,1.5);
	\draw (.5,.75) node{\tiny{$a_0$}};
	\draw (.2,-.2) node[above]{\tiny{$-$}};
	\draw (.8,-.2) node[above]{\tiny{$+$}};
	\draw (.8,1.7) node[below]{\tiny{$-$}};
	\draw (.2,1.7) node[below]{\tiny{$+$}};
	\end{scope}
	
	\begin{scope}[xshift = 5cm, yshift = 4.5cm]
	\draw (0,0) rectangle (1,1.5);
	\draw (.5,.75) node{\tiny{$b_0$}};
	\draw (.2,-.2) node[above]{\tiny{$-$}};
	\draw (.8,-.2) node[above]{\tiny{$+$}};
	\draw (.8,1.7) node[below]{\tiny{$-$}};
	\draw (.2,1.7) node[below]{\tiny{$+$}};
	\end{scope}
	
	\draw (5.2,1.5) arc (0:180:.7cm and .5cm);
	\draw (5.2,0) arc (0:-180:.7cm and .5cm);
	\draw (5.2,6) arc (0:180:.7cm and .5cm);
	\draw (5.2,4.5) arc (0:-180:.7cm and .5cm);
	\draw (3.2,1.5) -- (3.2,2.5);
	\draw (5.8,1.5) -- (5.8,2.5);
	\draw (3.2,4.5) -- (3.2,3.5);
	\draw (5.8,4.5) -- (5.8,3.5);
	\fill (4.5, 3) circle (.05cm);
	\fill (4.5,2.8) circle (.05cm);
	\fill (4.5,3.2) circle (.05cm);
	\draw (3.2,0) arc (0:-90:.5cm);
	\draw (.8,0) arc (180:270:.5cm);
	\draw (1.3,-.5) -- (2.7,-.5);
	\draw (.2,0) arc (0:-180:.5cm);
	\draw (.2,1.5) arc (0:180:.5cm);
	\draw [rounded corners = 1.5mm] (-.8,1.5) -- (-.8,1.25) -- (-1.8,.25) -- (-1.8,-.5);
	\draw [rounded corners = 1.5mm] (-.8,0) -- (-.8,.25) -- (-1.15,.6);
	\draw (-1.8,-.5) arc (180:270:.5cm);
	\draw (5.8,0) -- (5.8,-.5);
	\draw (5.8,-.5) arc (0:-90:.5cm);
	\draw (-1.3,-1) -- (5.3,-1);
	\draw (.8,1.5) -- (.8,6);
	\draw (.8,6) arc (180:90:.5cm);
	\draw (3.2,6) arc (0:90:.5cm);
	\draw (1.3,6.5) -- (2.7,6.5);
	\draw [rounded corners = 1.5mm] (-1.45,.9) -- (-1.8,1.25) -- (-1.8,6.5);
	\draw (-1.8,6.5) arc (180:90:.5cm);
	\draw (5.8,6) -- (5.8,6.5);
	\draw (5.8,6.5) arc(0:90:.5cm);
	\draw (-1.3,7) -- (5.3,7);


	\draw[ultra thick, ->] (7,3) -- (8,3);
	\draw (7.5,-2) node{Family 2};

\begin{scope}[xshift = 10cm]

	\draw (0,0) rectangle (1,1.5);
	\draw (.5,.75) node{\tiny{$c{-}1$}};
	\draw (.2,-.2) node[above]{\tiny{$-$}};
	\draw (.8,-.2) node[above]{\tiny{$+$}};
	\draw (.8,1.7) node[below]{\tiny{$-$}};
	\draw (.2,1.7) node[below]{\tiny{$+$}};
	
	\begin{scope}[xshift = 3cm]
	\draw (0,0) rectangle (1,1.5);
	\draw (.5,.75) node{\tiny{$a_n$}};
	\draw (.2,-.2) node[above]{\tiny{$-$}};
	\draw (.8,-.2) node[above]{\tiny{$+$}};
	\draw (.8,1.7) node[below]{\tiny{$-$}};
	\draw (.2,1.7) node[below]{\tiny{$+$}};
	\end{scope}
	
	\begin{scope}[xshift = 5cm]
	\draw (0,0) rectangle (1,1.5);
	\draw (.5,.75) node{\tiny{$b_n$}};
	\draw (.2,-.2) node[above]{\tiny{$-$}};
	\draw (.8,-.2) node[above]{\tiny{$+$}};
	\draw (.8,1.7) node[below]{\tiny{$-$}};
	\draw (.2,1.7) node[below]{\tiny{$+$}};
	\end{scope}
	
	\begin{scope}[xshift = 3cm, yshift = 4.5cm]
	\draw (0,0) rectangle (1,1.5);
	\draw (.5,.75) node{\tiny{$a_0$}};
	\draw (.2,-.2) node[above]{\tiny{$-$}};
	\draw (.8,-.2) node[above]{\tiny{$+$}};
	\draw (.8,1.7) node[below]{\tiny{$-$}};
	\draw (.2,1.7) node[below]{\tiny{$+$}};
	\end{scope}
	
	\begin{scope}[xshift = 5cm, yshift = 4.5cm]
	\draw (0,0) rectangle (1,1.5);
	\draw (.5,.75) node{\tiny{$b_0$}};
	\draw (.2,-.2) node[above]{\tiny{$-$}};
	\draw (.8,-.2) node[above]{\tiny{$+$}};
	\draw (.8,1.7) node[below]{\tiny{$-$}};
	\draw (.2,1.7) node[below]{\tiny{$+$}};
	\end{scope}
	
	\draw (5.2,1.5) arc (0:180:.7cm and .5cm);
	\draw (5.2,0) arc (0:-180:.7cm and .5cm);
	\draw (5.2,6) arc (0:180:.7cm and .5cm);
	\draw (5.2,4.5) arc (0:-180:.7cm and .5cm);
	\draw (3.2,1.5) -- (3.2,2.5);
	\draw (5.8,1.5) -- (5.8,2.5);
	\draw (3.2,4.5) -- (3.2,3.5);
	\draw (5.8,4.5) -- (5.8,3.5);
	\fill (4.5, 3) circle (.05cm);
	\fill (4.5,2.8) circle (.05cm);
	\fill (4.5,3.2) circle (.05cm);
	\draw (3.2,0) arc (0:-90:.5cm);
	\draw (.8,0) arc (180:270:.5cm);
	\draw (1.3,-.5)  -- (1.6,-.5);
	\draw (2,-.5) -- (2.7,-.5);
	\draw (5.8,0) -- (5.8,-.5);
	\draw (5.8,-.5) arc (0:-90:.5cm);
	\draw (2.3,-1) -- (5.3,-1);
	\draw (.2,1.5) -- (.2,6);
	\draw (.2,6) arc (180:90:.5cm);
	\draw (3.2,6) arc (0:90:.5cm);
	\draw (.7,6.5) -- (2.7,6.5);
	\draw (5.8,6) -- (5.8,6.5);
	\draw (5.8,6.5) arc(0:90:.5cm);
	\draw (-.3,7) -- (5.3,7);
	\draw (.8,1.5) arc (180:0:.5cm);
	\draw (1.8,1.5) -- (1.8,-.5);
	\draw (1.8,-.5) arc (180:270:.5cm);
	\draw (.2,0) arc (0:-180:.5cm);
	\draw (-.8,0) -- (-.8,6.5);
	\draw (-.8,6.5) arc (180:90:.5cm);

	\begin{scope}[dashed, red, thick, rounded corners = 1mm]
	\draw (-.6, 1) -- (-.6,0) -- (-.2,0) -- (-.2,3) -- (.6,3.8) -- (.6,6) -- (2.8,6) -- (2.8,-.2) -- (2.1,-.5) -- (2.1,-.9) -- (4.4,-.9) -- (4.6,-1.1) -- (6.2,-1.1) -- (6.2,6) -- (5.4,6.8) -- (-.6,6.8) -- (-.6,1);
	\draw (1.2,.7) -- (1.2,-.3) -- (1.6,-.3) -- (1.6,1.6) -- (1.2,1.6) -- (1.2,.7);
	\draw (4.2,1) -- (4.2,-.2) -- (4.8,-.2) -- (4.8,1.7) -- (4.2,1.7) -- (4.2,1);
	\draw[yshift = 4.5cm] (4.2,1) -- (4.2,-.2) -- (4.8,-.2) -- (4.8,1.7) -- (4.2,1.7) -- (4.2,1);
	\end{scope}

\end{scope}

\end{scope}

\begin{scope}[yshift = 0cm]
	
	\begin{scope}[xshift = 3cm]
	\draw (0,0) rectangle (1,1.5);
	\draw (.5,.75) node{\tiny{$a_n$}};
	\draw (.2,-.2) node[above]{\tiny{$-$}};
	\draw (.8,-.2) node[above]{\tiny{$+$}};
	\draw (.8,1.7) node[below]{\tiny{$-$}};
	\draw (.2,1.7) node[below]{\tiny{$+$}};
	\end{scope}
	
	\begin{scope}[xshift = 5cm]
	\draw (0,0) rectangle (1,1.5);
	\draw (.5,.75) node{\tiny{$b_n$}};
	\draw (.2,-.2) node[above]{\tiny{$-$}};
	\draw (.8,-.2) node[above]{\tiny{$+$}};
	\draw (.8,1.7) node[below]{\tiny{$-$}};
	\draw (.2,1.7) node[below]{\tiny{$+$}};
	\end{scope}
	
	\begin{scope}[xshift = 3cm, yshift = 4.5cm]
	\draw (0,0) rectangle (1,1.5);
	\draw (.5,.75) node{\tiny{$a_0$}};
	\draw (.2,-.2) node[above]{\tiny{$-$}};
	\draw (.8,-.2) node[above]{\tiny{$+$}};
	\draw (.8,1.7) node[below]{\tiny{$-$}};
	\draw (.2,1.7) node[below]{\tiny{$+$}};
	\end{scope}
	
	\begin{scope}[xshift = 5cm, yshift = 4.5cm]
	\draw (0,0) rectangle (1,1.5);
	\draw (.5,.75) node{\tiny{$b_0$}};
	\draw (.2,-.2) node[above]{\tiny{$-$}};
	\draw (.8,-.2) node[above]{\tiny{$+$}};
	\draw (.8,1.7) node[below]{\tiny{$-$}};
	\draw (.2,1.7) node[below]{\tiny{$+$}};
	\end{scope}
	
	\begin{scope}[xshift = 5cm, yshift = -3cm]
	\draw (0,0) rectangle (1,1.5);
	\draw (.5,.75) node{\tiny{$c$}};
	\draw (.2,-.2) node[above]{\tiny{$-$}};
	\draw (.8,-.2) node[above]{\tiny{$+$}};
	\draw (.8,1.7) node[below]{\tiny{$-$}};
	\draw (.2,1.7) node[below]{\tiny{$+$}};
	\end{scope}

	\draw (5.2,1.5) arc (0:180:.7cm and .5cm);
	\draw (5.2,0) arc (0:-180:.7cm and .5cm);
	\draw (5.2,6) arc (0:180:.7cm and .5cm);
	\draw (5.2,4.5) arc (0:-180:.7cm and .5cm);
	\draw (3.2,1.5) -- (3.2,2.5);
	\draw (5.8,1.5) -- (5.8,2.5);
	\draw (3.2,4.5) -- (3.2,3.5);
	\draw (5.8,4.5) -- (5.8,3.5);
	\fill (4.5, 3) circle (.05cm);
	\fill (4.5,2.8) circle (.05cm);
	\fill (4.5,3.2) circle (.05cm);
	\draw (3.2,0) arc (0:-90:.5cm);
	\draw[rounded corners = 1.5mm](-.2,-3.5)-- (-.2,-1.5) -- (.8,-.5) -- (2.7,-.5);
	\draw[rounded corners = 1.5mm] (1.1,-1.5) -- (.8,-1.5) -- (.5,-1.2);
	\draw[rounded corners = 1.5mm] (.1,-.8) -- (-.2,-.5) -- (-.2,6.5);
	\draw (-.2,6.5) arc (180:90:.5cm);
	\draw (-.2,-3.5) arc (180:270:.5cm);
	\draw (.3,-4) -- (5.3,-4);
	\draw (5.3,-4) arc (-90:0:.5cm);
	\draw (5.8,-3.5) -- (5.8,-3);
	\draw (5.8,0) -- (5.8,-1.5);
	\draw (5.2, -1.5) arc (0:180:.5cm);
	\draw (4.2,-1.5) arc (0:-90:.5cm);
	\draw (3.7,-2) -- (2.2,-2);
	\draw (5.2,-3) arc (0:-90:.5cm);
	\draw (4.7,-3.5) -- (2.5,-3.5);
	\draw (2.5,-3.5) arc (-90:-180:.5cm);
	\draw (2,-3) -- (2,-2);
	\draw (2,-2) arc (0:90:.5cm);
	\draw (1.8,-2) arc (-90:-180:.5cm);
	\draw (1.3,-1.5) -- (1.3, -.7);
	\draw (1.3,-.3) -- (1.3,6);
	\draw (1.3,6) arc (180:90:.5cm);
	\draw (3.2,6) arc (0:90:.5cm);
	\draw (1.8,6.5) -- (2.7,6.5);
	\draw (5.8,6) -- (5.8,6.5);
	\draw (5.8,6.5) arc(0:90:.5cm);
	\draw (.3,7) -- (5.3,7);
	
	\draw[ultra thick, ->] (7,3) -- (8,3);
	\draw (7.5,-5) node{Family 1};

\begin{scope}[xshift = 8cm]

	\begin{scope}[xshift = 3cm]
	\draw (0,0) rectangle (1,1.5);
	\draw (.5,.75) node{\tiny{$a_n$}};
	\draw (.2,-.2) node[above]{\tiny{$-$}};
	\draw (.8,-.2) node[above]{\tiny{$+$}};
	\draw (.8,1.7) node[below]{\tiny{$-$}};
	\draw (.2,1.7) node[below]{\tiny{$+$}};
	\end{scope}
	
	\begin{scope}[xshift = 5cm]
	\draw (0,0) rectangle (1,1.5);
	\draw (.5,.75) node{\tiny{$b_n$}};
	\draw (.2,-.2) node[above]{\tiny{$-$}};
	\draw (.8,-.2) node[above]{\tiny{$+$}};
	\draw (.8,1.7) node[below]{\tiny{$-$}};
	\draw (.2,1.7) node[below]{\tiny{$+$}};
	\end{scope}
	
	\begin{scope}[xshift = 3cm, yshift = 4.5cm]
	\draw (0,0) rectangle (1,1.5);
	\draw (.5,.75) node{\tiny{$a_0$}};
	\draw (.2,-.2) node[above]{\tiny{$-$}};
	\draw (.8,-.2) node[above]{\tiny{$+$}};
	\draw (.8,1.7) node[below]{\tiny{$-$}};
	\draw (.2,1.7) node[below]{\tiny{$+$}};
	\end{scope}
	
	\begin{scope}[xshift = 5cm, yshift = 4.5cm]
	\draw (0,0) rectangle (1,1.5);
	\draw (.5,.75) node{\tiny{$b_0$}};
	\draw (.2,-.2) node[above]{\tiny{$-$}};
	\draw (.8,-.2) node[above]{\tiny{$+$}};
	\draw (.8,1.7) node[below]{\tiny{$-$}};
	\draw (.2,1.7) node[below]{\tiny{$+$}};
	\end{scope}
	
	\begin{scope}[xshift = 5cm, yshift = -3cm]
	\draw (0,0) rectangle (1,1.5);
	\draw (.5,.75) node{\tiny{$c{-}1$}};
	\draw (.2,-.2) node[above]{\tiny{$-$}};
	\draw (.8,-.2) node[above]{\tiny{$+$}};
	\draw (.8,1.7) node[below]{\tiny{$-$}};
	\draw (.2,1.7) node[below]{\tiny{$+$}};
	\end{scope}

	\draw (5.2,1.5) arc (0:180:.7cm and .5cm);
	\draw (5.2,0) arc (0:-180:.7cm and .5cm);
	\draw (5.2,6) arc (0:180:.7cm and .5cm);
	\draw (5.2,4.5) arc (0:-180:.7cm and .5cm);
	\draw (3.2,1.5) -- (3.2,2.5);
	\draw (5.8,1.5) -- (5.8,2.5);
	\draw (3.2,4.5) -- (3.2,3.5);
	\draw (5.8,4.5) -- (5.8,3.5);
	\fill (4.5, 3) circle (.05cm);
	\fill (4.5,2.8) circle (.05cm);
	\fill (4.5,3.2) circle (.05cm);
	\draw (1.1,-1.5) arc (-90:-180:.5cm);
	\draw (.6,-1) -- (.6,6.5);
	\draw (.6,6.5) arc (180:90:.5cm);
	\draw (2.5,-4) -- (5.3,-4);
	\draw (2.5,-4) arc (-90:-180:.5cm);
	\draw (5.3,-4) arc (-90:0:.5cm);
	\draw (5.8,-3.5) -- (5.8,-3);
	\draw (5.8,0) -- (5.8,-1.5);
	\draw (5.2, -1.5) arc (0:180:.5cm);
	\draw (4.2,-1.5) arc (0:-90:.5cm);
	\draw (3.7,-2) -- (3.4,-2);
	\draw (3,-2) --  (2.2,-2);
	\draw (5.2,-3) arc (0:-90:.5cm);
	\draw (4.7,-3.5) -- (3.7,-3.5);
	\draw (3.7,-3.5) arc (-90:-180:.5cm);
	\draw (3.2,-3) -- (3.2,0);
	\draw (2,-3.5) -- (2,-2);
	\draw (2,-2) arc (0:90:.5cm);
	\draw (1.8,-2) arc (-90:-180:.5cm);
	\draw (1.3,-1.5) -- (1.3, -.7);
	\draw (1.3,-.7) -- (1.3,6);
	\draw (1.3,6) arc (180:90:.5cm);
	\draw (3.2,6) arc (0:90:.5cm);
	\draw (1.8,6.5) -- (2.7,6.5);
	\draw (5.8,6) -- (5.8,6.5);
	\draw (5.8,6.5) arc(0:90:.5cm);
	\draw (1.1,7) -- (5.3,7);
	
	
	\begin{scope}[red, thick, rounded corners = 1mm,dashed]
	\draw[solid] (1.65,-1.725) ellipse (.15cm and .1cm);
	\draw (2.2,-3) -- (2.2,-3.8) -- (3.8,-3.8) -- (4.2,-4.2) -- (6.2,-4.2) -- (6.2,6) -- (5.4,6.8) -- (.8,6.8) -- (.8,-1.1) -- (1.1,-1.1) -- (1.1,2.5) -- (1.5,3) -- (1.5,6.2) -- (2.8,6.2) -- (2.8,-1.8) -- (2.2,-2.2) -- (2.2,-3);
	\draw (3.4,-3) -- (3.4,-3.3) -- (4.8,-3.3) -- (4.8,-1.2) -- (4.4,-1.2) -- (4.4,-2.2) -- (3.4,-2.2) -- (3.4,-3);
	\draw (4.2,1) -- (4.2,-.2) -- (4.8,-.2) -- (4.8,1.7) -- (4.2,1.7) -- (4.2,1);
	\draw[yshift = 4.5cm] (4.2,1) -- (4.2,-.2) -- (4.8,-.2) -- (4.8,1.7) -- (4.2,1.7) -- (4.2,1);
	
	\end{scope}

\end{scope}

\end{scope}

\begin{scope}[yshift = -23cm]
	
	\begin{scope}[xshift = 3cm]
	\draw (0,0) rectangle (1,1.5);
	\draw (.5,.75) node{\tiny{$a_n$}};
	\draw (.2,-.2) node[above]{\tiny{$-$}};
	\draw (.8,-.2) node[above]{\tiny{$+$}};
	\draw (.8,1.7) node[below]{\tiny{$-$}};
	\draw (.2,1.7) node[below]{\tiny{$+$}};
	\end{scope}
	
	\begin{scope}[xshift = 5cm]
	\draw (0,0) rectangle (1,1.5);
	\draw (.5,.75) node{\tiny{$b_n$}};
	\draw (.2,-.2) node[above]{\tiny{$-$}};
	\draw (.8,-.2) node[above]{\tiny{$+$}};
	\draw (.8,1.7) node[below]{\tiny{$-$}};
	\draw (.2,1.7) node[below]{\tiny{$+$}};
	\end{scope}
	
	\begin{scope}[xshift = 3cm, yshift = 4.5cm]
	\draw (0,0) rectangle (1,1.5);
	\draw (.5,.75) node{\tiny{$a_0$}};
	\draw (.2,-.2) node[above]{\tiny{$-$}};
	\draw (.8,-.2) node[above]{\tiny{$+$}};
	\draw (.8,1.7) node[below]{\tiny{$-$}};
	\draw (.2,1.7) node[below]{\tiny{$+$}};
	\end{scope}
	
	\begin{scope}[xshift = 5cm, yshift = 4.5cm]
	\draw (0,0) rectangle (1,1.5);
	\draw (.5,.75) node{\tiny{$b_0$}};
	\draw (.2,-.2) node[above]{\tiny{$-$}};
	\draw (.8,-.2) node[above]{\tiny{$+$}};
	\draw (.8,1.7) node[below]{\tiny{$-$}};
	\draw (.2,1.7) node[below]{\tiny{$+$}};
	\end{scope}
	
	\begin{scope}[xshift = 3cm, yshift = -3cm]
	\draw (0,0) rectangle (1,1.5);
	\draw (.5,.75) node{\tiny{$c$}};
	\draw (.2,-.2) node[above]{\tiny{$-$}};
	\draw (.8,-.2) node[above]{\tiny{$+$}};
	\draw (.8,1.7) node[below]{\tiny{$-$}};
	\draw (.2,1.7) node[below]{\tiny{$+$}};
	\end{scope}

	\draw (5.2,1.5) arc (0:180:.7cm and .5cm);
	\draw (5.2,6) arc (0:180:.7cm and .5cm);
	\draw (5.2,4.5) arc (0:-180:.7cm and .5cm);
	\draw (3.2,1.5) -- (3.2,2.5);
	\draw (5.8,1.5) -- (5.8,2.5);
	\draw (3.2,4.5) -- (3.2,3.5);
	\draw (5.8,4.5) -- (5.8,3.5);
	\fill (4.5, 3) circle (.05cm);
	\fill (4.5,2.8) circle (.05cm);
	\fill (4.5,3.2) circle (.05cm);
	\draw (3.2,0) arc (0:-90:.5cm);
	\draw[rounded corners = 1.5mm](-.2,-3.5)-- (-.2,-1.5) -- (.8,-.5) -- (2.7,-.5);
	\draw[rounded corners = 1.5mm] (.1,-.8) -- (-.2,-.5) -- (-.2,6.5);
	\draw (-.2,6.5) arc (180:90:.5cm);
	\draw (-.2,-3.5) arc (180:270:.5cm);
	\draw (.3,-4) -- (5.3,-4);
	\draw (5.8,0) -- (5.8,-3.5);
	\draw (5.8,-3.5) arc (0:-90:.5cm);
	\draw (3.8,0) -- (3.8,-1.5);
	\draw (3.8,-3) arc (-180:0:.7cm and .5cm);
	\draw (5.2,-3) -- (5.2,0);
	\draw (1.3,-3) -- (1.3, -.7);
	\draw (1.3,-3) arc (180:270:.5cm);
	\draw (1.8,-3.5) -- (2.7,-3.5);
	\draw (3.2,-3) arc (0:-90:.5cm);
	\draw (1.3,-.3) -- (1.3,6);
	\draw (1.3,6) arc (180:90:.5cm);
	\draw (3.2,6) arc (0:90:.5cm);
	\draw (3.2,-1.5) arc (0:180:.5cm);
	\draw (2.2,-1.5) arc (0:-90:.5cm);
	\draw (1.7,-2) -- (1.5,-2);
	\draw (1.8,6.5) -- (2.7,6.5);
	\draw (5.8,6) -- (5.8,6.5);
	\draw (5.8,6.5) arc(0:90:.5cm);
	\draw (.3,7) -- (5.3,7);
	\draw (1.1,-2) arc (-90:-180:.3cm);
	\draw[rounded corners = 1.5mm] (.8,-1.7) -- (.8,-1.5) -- (.5,-1.2);
	

	\draw[ultra thick, ->] (7,3) -- (8,3);
	\draw (7.5,-5) node{Family 3};

\begin{scope}[xshift = 9cm]

	\begin{scope}[xshift = 3cm]
	\draw (0-.15,0) rectangle (1.15,1.5);
	\draw (.5,.75) node{\tiny{$a_n{+}1$}};
	\draw (.2,-.2) node[above]{\tiny{$-$}};
	\draw (.8,-.2) node[above]{\tiny{$+$}};
	\draw (.8,1.7) node[below]{\tiny{$-$}};
	\draw (.2,1.7) node[below]{\tiny{$+$}};
	\end{scope}
	
	\begin{scope}[xshift = 5cm]
	\draw (-.15,0) rectangle (1.15,1.5);
	\draw (.5,.75) node{\tiny{$b_n{-}1$}};
	\draw (.2,-.2) node[above]{\tiny{$-$}};
	\draw (.8,-.2) node[above]{\tiny{$+$}};
	\draw (.8,1.7) node[below]{\tiny{$-$}};
	\draw (.2,1.7) node[below]{\tiny{$+$}};
	\end{scope}
	
	\begin{scope}[xshift = 3cm, yshift = 4.5cm]
	\draw (0,0) rectangle (1,1.5);
	\draw (.5,.75) node{\tiny{$a_0$}};
	\draw (.2,-.2) node[above]{\tiny{$-$}};
	\draw (.8,-.2) node[above]{\tiny{$+$}};
	\draw (.8,1.7) node[below]{\tiny{$-$}};
	\draw (.2,1.7) node[below]{\tiny{$+$}};
	\end{scope}
	
	\begin{scope}[xshift = 5cm, yshift = 4.5cm]
	\draw (0,0) rectangle (1,1.5);
	\draw (.5,.75) node{\tiny{$b_0$}};
	\draw (.2,-.2) node[above]{\tiny{$-$}};
	\draw (.8,-.2) node[above]{\tiny{$+$}};
	\draw (.8,1.7) node[below]{\tiny{$-$}};
	\draw (.2,1.7) node[below]{\tiny{$+$}};
	\end{scope}
	
	\begin{scope}[xshift = 3cm, yshift = -3cm]
	\draw (0,0) rectangle (1,1.5);
	\draw (.5,.75) node{\tiny{$c$}};
	\draw (.2,-.2) node[above]{\tiny{$-$}};
	\draw (.8,-.2) node[above]{\tiny{$+$}};
	\draw (.8,1.7) node[below]{\tiny{$-$}};
	\draw (.2,1.7) node[below]{\tiny{$+$}};
	\end{scope}

	\draw (3.8,0) arc (-180:0:.7cm and .5cm);
	
	\draw[rounded corners = 1.5mm] (3.2,0) -- (3.2,-.5) -- (3.8,-1) -- (3.8,-1.5);

	\draw (5.2,1.5) arc (0:180:.7cm and .5cm);
	\draw (5.2,6) arc (0:180:.7cm and .5cm);
	\draw (5.2,4.5) arc (0:-180:.7cm and .5cm);
	\draw (3.2,1.5) -- (3.2,2.5);
	\draw (5.8,1.5) -- (5.8,2.5);
	\draw (3.2,4.5) -- (3.2,3.5);
	\draw (5.8,4.5) -- (5.8,3.5);
	\fill (4.5, 3) circle (.05cm);
	\fill (4.5,2.8) circle (.05cm);
	\fill (4.5,3.2) circle (.05cm);
	\draw (-.2,-.5) -- (-.2,6.5);
	\draw (-.2,6.5) arc (180:90:.5cm);
	\draw (5.8,0) -- (5.8,-3);
	\draw (3.8,-3) arc (-180:0:1cm and .5cm);
	\draw (1.3,-3) arc (180:270:.5cm);
	\draw (1.8,-3.5) -- (2.7,-3.5);
	\draw (3.2,-3) arc (0:-90:.5cm);
	\draw (2.7,-1) -- (1.5,-1);
	\draw (1.1, -1) -- (.3,-1);
	\draw (.3,-1) arc (-90:-180:.5cm);
	\draw (1.3,-3) -- (1.3,6);
	\draw (1.3,6) arc (180:90:.5cm);
	\draw (3.2,6) arc (0:90:.5cm);
	\draw (3.2,-1.5) arc (0:90:.5cm);
	\draw (1.8,6.5) -- (2.7,6.5);
	\draw (5.8,6) -- (5.8,6.5);
	\draw (5.8,6.5) arc(0:90:.5cm);
	\draw (.3,7) -- (5.3,7);
	
	\begin{scope}[thick, red, dashed, rounded corners = 1mm]
	\draw (6.3,1) -- (6.3,-.1) -- (5.6,-.8) -- (5.6,-3.2) -- (4.2,-3.2) -- (4.2,-.6) -- (2.7,-.6) -- (2.7,6.2) -- (1.5,6.2) -- (1.5,3) -- (1.1,2.6) -- (1.1,-.8) -- (0,-.8) -- (0,6.8) -- (5.6,6.8) -- (6.3,6.1) -- (6.3,1);
	\draw[yshift = 4.5cm] (4.2,1) -- (4.2,-.2) -- (4.8,-.2) -- (4.8,1.7) -- (4.2,1.7) -- (4.2,1);
	\draw (4.3,1) -- (4.3,-.2) -- (4.7,-.2) -- (4.7,1.7) -- (4.3,1.7) -- (4.3,1);
	\draw (2.8, -2) -- (2.8,-3.2) -- (1.5,-3.2) -- (1.5,-1.2) -- (2.8,-1.2) -- (2.8,-2);
	\end{scope}

\end{scope}

\end{scope}
	
	\end{tikzpicture}$$
\caption{The diagrams on the left are associated to the graphs $\overline{G}$ in Families 1, 2, and 3. In each case, an isotopy yields a $B$-adequate diagram. The dashed curves are the all-$B$ states.}
\label{figure:F123}
\end{figure}
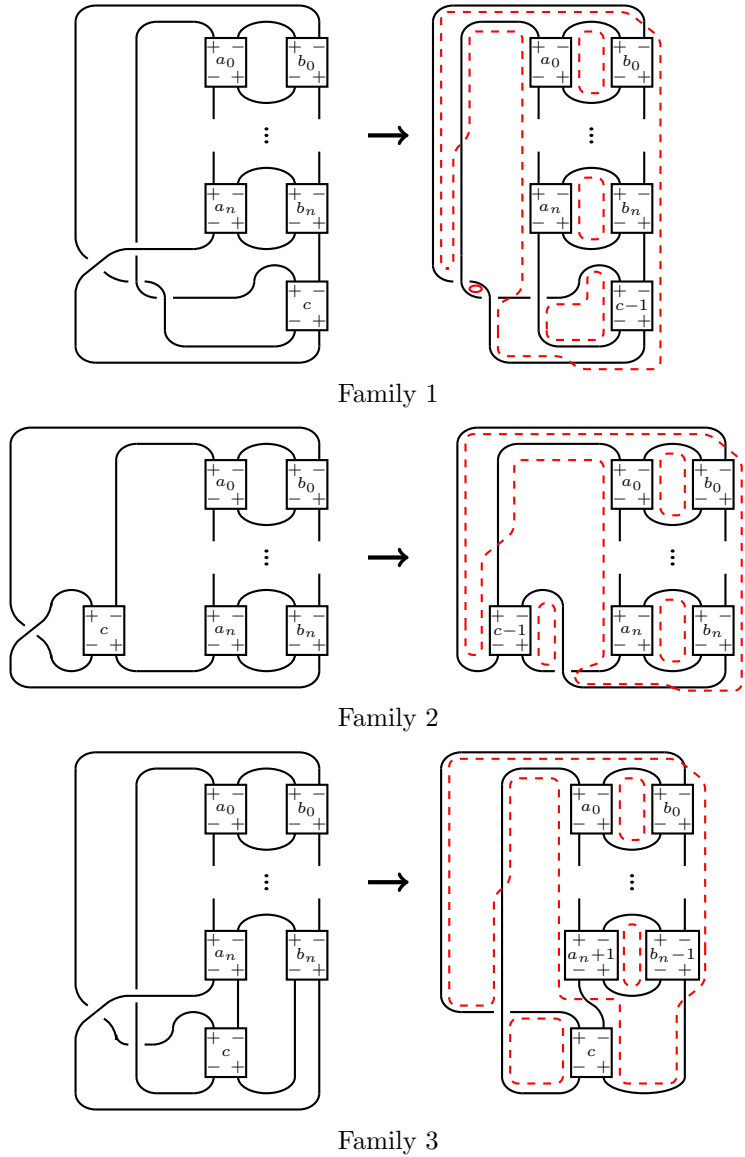

\begin{lemma}
\label{lemma:F4-7}
Let $D$ be a strongly reduced almost alternating diagram. Suppose that $\alpha_{c-4}=0$, $P=1$, and $\overline{P}=2$. After possibly relabeling vertices $v_1$ and $v_2$, the checkerboard graph $\overline{G}$ belongs to one of the four families in Figure \ref{figure:PD22}. Moreover, every link whose checkerboard graph is in any of these families is alternating.
\end{lemma}

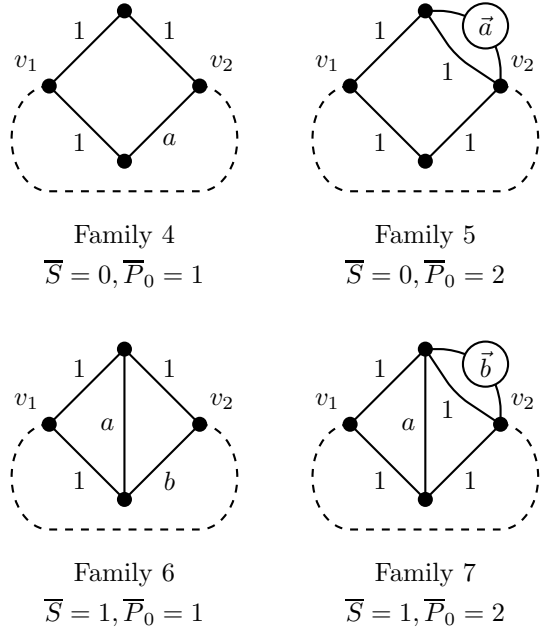
\begin{figure}[h]
$$\begin{tikzpicture}[thick]


	\fill (0,0) circle (.1cm);
	\fill (1,1) circle (.1cm);
	\fill (2,0) circle (.1cm);
	\fill (1,-1) circle (.1cm);
	\draw (-.3,.3) node{$v_1$};
	\draw (2.3,0.3) node{$v_2$};
	
	\draw (0,0) -- (1,1) -- (2,0) -- (1,-1) -- (0,0);
	\draw (.4,.5) node[above]{1};
	\draw (.4,-.5) node[below]{1};
	\draw (1.6,.5) node[above]{1};
	\draw (1.6,-.5) node[below]{$a$};
	\draw[dashed] (0,0) arc (90:270:.5cm and .7cm);
	\draw[dashed] (2,0) arc (90:-90:.5cm and .7cm);
	\draw[dashed] (0,-1.4) -- (2,-1.4);
	
	\draw (1,-2) node{Family 4};
	\draw (1,-2.5) node{$\overline{S}=0, \overline{P}_{0}=1$};
	
\begin{scope}[xshift = 4cm]

	\fill (0,0) circle (.1cm);
	\fill (1,1) circle (.1cm);
	\fill (2,0) circle (.1cm);
	\fill (1,-1) circle (.1cm);
	\draw (-.3,.3) node{$v_1$};
	\draw (2.3,0.3) node{$v_2$};
	
	\draw (0,0) -- (1,1);
	\draw (2,0) -- (1,-1) -- (0,0);
	\draw [rounded corners = 2mm] (1,1) -- (1.4,.4) -- (2,0);
	\draw (.4,.5) node[above]{1};
	\draw (.4,-.5) node[below]{1};
	\draw (1.3,.2) node{1};
	\draw (1.6,-.5) node[below]{$1$};
	\draw[rounded corners = 2mm] (1,1) -- (1.4,1) -- (1.8,.8);
	\draw[rounded corners = 2mm] (2,0) -- (2,.4) -- (1.8,.8);
	\fill[white] (1.8,.8) circle(.3cm);
	\draw (1.8,.8) circle(.3cm);
	\draw (1.8,.8) node{$\vec{a}$};
	\draw[dashed] (0,0) arc (90:270:.5cm and .7cm);
	\draw[dashed] (2,0) arc (90:-90:.5cm and .7cm);
	\draw[dashed] (0,-1.4) -- (2,-1.4);
	
	\draw (1,-2) node{Family 5};
	\draw (1,-2.5) node{$\overline{S}=0, \overline{P}_{0}=2$};

\end{scope}

\begin{scope}[yshift = -4.5 cm]

	\fill (0,0) circle (.1cm);
	\fill (1,1) circle (.1cm);
	\fill (2,0) circle (.1cm);
	\fill (1,-1) circle (.1cm);
	\draw (-.3,.3) node{$v_1$};
	\draw (2.3,0.3) node{$v_2$};
	
	\draw (0,0) -- (1,1) -- (2,0) -- (1,-1) -- (0,0);
	\draw (1,1) -- (1,-1);
	\draw (.4,.5) node[above]{1};
	\draw (.4,-.5) node[below]{1};
	\draw (1.6,.5) node[above]{1};
	\draw (1.6,-.5) node[below]{$b$};
	\draw (1,0) node[left]{$a$};
	\draw[dashed] (0,0) arc (90:270:.5cm and .7cm);
	\draw[dashed] (2,0) arc (90:-90:.5cm and .7cm);
	\draw[dashed] (0,-1.4) -- (2,-1.4);
	
	\draw (1,-2) node{Family 6};
	\draw (1,-2.5) node{$\overline{S}=1, \overline{P}_{0}=1$};

\end{scope}

\begin{scope}[xshift = 4cm, yshift = -4.5cm]

	\fill (0,0) circle (.1cm);
	\fill (1,1) circle (.1cm);
	\fill (2,0) circle (.1cm);
	\fill (1,-1) circle (.1cm);
	\draw (-.3,.3) node{$v_1$};
	\draw (2.3,0.3) node{$v_2$};
	
	\draw (0,0) -- (1,1);
	\draw (2,0) -- (1,-1) -- (0,0);
	\draw [rounded corners = 2mm] (1,1) -- (1.4,.4) -- (2,0);
	\draw (.4,.5) node[above]{1};
	\draw (.4,-.5) node[below]{1};
	\draw (1.3,.2) node{1};
	\draw (1.6,-.5) node[below]{$1$};
	\draw[rounded corners = 2mm] (1,1) -- (1.4,1) -- (1.8,.8);
	\draw[rounded corners = 2mm] (2,0) -- (2,.4) -- (1.8,.8);
	\fill[white] (1.8,.8) circle(.3cm);
	\draw (1.8,.8) circle(.3cm);
	\draw (1.8,.8) node{$\vec{b}$};
	\draw (1,1) -- (1,-1);
	\draw (1,0) node[left]{$a$};
	\draw[dashed] (0,0) arc (90:270:.5cm and .7cm);
	\draw[dashed] (2,0) arc (90:-90:.5cm and .7cm);
	\draw[dashed] (0,-1.4) -- (2,-1.4);
	
	\draw (1,-2) node{Family 7};
	\draw (1,-2.5) node{$\overline{S}=1, \overline{P}_{0}=2$};

\end{scope}

\end{tikzpicture}$$
\caption{Four families of $\overline{G}$ when $\alpha_{c-4}=0$, $P=1$, and $\overline{P}=2$.}
\label{figure:PD22}
\end{figure}

\begin{proof}
Since $\alpha_{c-4}=0$, $P=1$, and $\overline{P}=2$, we have that $\overline{P}_2=\overline{Q}=0$ and 
$$\overline{\beta_1}=\overline{P}_0+\overline{S}+1.$$
Since $\overline{P}=2$, we have that $\overline{P_0}=0,1$, or 2. Furthermore, since there are two paths of length two in each $K_4$ subgraph containing $v_1$ and $v_2$, it follows that $\overline{P}=2$ implies that $\overline{S}=0$ or $1$.

Suppose that $\overline{S}=0$. Then $\overline{\beta}_1= \overline{P}_0 + 1$. Since $\overline{P}_0\leq \overline{P}=2$ and $2\leq \overline{\beta}_1$, we have two cases: either $\overline{\beta}_1=2$ and $\overline{P}_0=1$ or $\overline{\beta}_1=3$ and $\overline{P}_0=2$. In the former case, we obtain Family 4, and in the latter case we obtain Family 5.

Suppose $\overline{S}=1$. Then $\overline{\beta}_1=\overline{P}_0+2$. Since $\overline{S}=1$, we have that $3\leq \overline{\beta}_1$, and thus $3\leq \overline{P}_0+2$. Again there are two cases: either $\overline{\beta}_1=3$ and $\overline{P}_0=1$ or $\overline{\beta}_1 = 4$ and $\overline{P}_0=2$. In the former case, we obtain Family 6, and in the latter case we obtain Family 7.

The diagrams in the left column of Figure \ref{figure:Family4567} are the four families of link diagrams whose checkerboard graphs $\overline{G}$ are shown in Figure \ref{figure:PD22}. For each of the four families, Figure \ref{figure:Family4567} shows an isotopy to an alternating link. In the case of Family 7, the diagram is non-alternating; however since it is a connected sum of alternating diagrams, it follows that the link is alternating. Since the links associated to the checkerboard graphs in these families are alternating, they are not almost alternating.
\end{proof}
\begin{figure}[h]
$$\begin{tikzpicture}[thick, scale =.57]
	
	\draw (2,3.8) -- (2,1);
	\draw (2,1) arc (180:270:.5cm);
	\draw (2.5,.5) -- (2.7,.5);
	\draw (2.7,.5) arc (-90:0:.5cm);
	\draw (2.2,3) -- (2.7,3);
	\draw (2.7,3) arc (90:0:.5cm);
	\draw (3.8,2.5) -- (3.8,4.2);
	\draw (3.8,4.2) arc (0:90:.5cm);
	\draw (3.3,4.7) -- (2.5,4.7);
	\draw (2.5,4.7) arc (90:180:.5cm);
	\draw[rounded corners = 2mm] (3.6,4) -- (1,4) -- (0,3) -- (0,.5);
	\draw (0,.5) arc (-180:-90:.5cm);
	\draw[rounded corners = 2mm] (1.8,3) -- (1,3) -- (.7,3.3);
	\draw [rounded corners = 2mm] (.3,3.7) -- (0,4) -- (0,4.7);
	\draw (0,4.7) arc (180:90:.5cm);
	\draw (4,4) arc (-90:90:.6cm);
	\draw (4,5.2) -- (.5,5.2);
	\draw (.5,0) -- (3.3,0);
	\draw (3.8,1) -- (3.8,.5);
	\draw (3.8,.5) arc (0:-90:.5cm);
	
	\begin{scope}[xshift = 3cm, yshift = 1cm]
	\draw (0,0) rectangle (1,1.5);
	\draw (.5,.75) node{\tiny{$a$}};
	\draw (.2,0) node[above]{\tiny{$-$}};
	\draw (.8,0) node[above]{\tiny{$+$}};
	\draw (.8,1.5) node[below]{\tiny{$-$}};
	\draw (.2,1.5) node[below]{\tiny{$+$}};
	\end{scope}
	

\begin{scope}[xshift = 5cm]
	
	\draw (2,4.2) -- (2,1);
	\draw (2,1) arc (180:270:.5cm);
	\draw (2.5,.5) -- (2.7,.5);
	\draw (2.7,.5) arc (-90:0:.5cm);
	\draw (2.2,3) -- (2.7,3);
	\draw (2.7,3) arc (90:0:.5cm);
	\draw (3.8,2.5) -- (3.8,3);
	\draw (3.8,3.4)-- (3.8,4.2);
	\draw (3.8,4.2) arc (0:90:.5cm);
	\draw (3.3,4.7) -- (2.5,4.7);
	\draw (2.5,4.7) arc (90:180:.5cm);
	\draw (3.6,4) arc (90:270:.4cm);
	\draw (3.6,3.2) -- (4.3,3.2);
	\draw (4.3,3.2) arc (90:0:.5cm);
	\draw (4.8,2.7) -- (4.8,1);
	\draw (3.8,1) arc (-180:0:.5cm);
	
	\draw (1.8,3) arc (-90:-180:.5cm);
	\draw (1.3,3.5) -- (1.3,4.7);
	\draw (1.3,4.7) arc (180:90:.5cm);
	\draw (4,4) arc (-90:90:.6cm);
	\draw (4,5.2) -- (1.8,5.2);
	
	\begin{scope}[xshift = 3cm, yshift = 1cm]
	\draw (0,0) rectangle (1,1.5);
	\draw (.5,.75) node{\tiny{$a$}};
	\draw (.2,0) node[above]{\tiny{$-$}};
	\draw (.8,0) node[above]{\tiny{$+$}};
	\draw (.8,1.5) node[below]{\tiny{$-$}};
	\draw (.2,1.5) node[below]{\tiny{$+$}};
	\end{scope}
	
\end{scope}


\begin{scope}[xshift = 10cm]

	\draw (3.2,2.5) arc (0:180:.5cm);
	\draw (3.8,2.5) arc (180:0:.5cm);
	\draw (3.2,1) arc (0:-180:.5cm);
	\draw (3.8,1) arc (-180:0:.5cm);
	\draw (2.2,2.5) -- (2.2,1);
	\draw (4.8,2.5) -- (4.8,1);
	
	\begin{scope}[xshift = 3cm, yshift = 1cm]
	\draw (0,0) rectangle (1,1.5);
	\draw (.5,.75) node{\tiny{$a{-}1$}};
	\draw (.2,0) node[above]{\tiny{$-$}};
	\draw (.8,0) node[above]{\tiny{$+$}};
	\draw (.8,1.5) node[below]{\tiny{$-$}};
	\draw (.2,1.5) node[below]{\tiny{$+$}};
	\end{scope}

\end{scope}

\draw[ultra thick, ->] (5,2) -- (6,2);
\draw[ultra thick, ->] (10.5,2) -- (11.5,2);
\draw (7.75,-1) node{Family 4};


\begin{scope}[yshift = -6cm, xshift = 1cm]

	\draw[ultra thick, ->] (3.75,1.5) -- (4.75,1.5);
	\draw[ultra thick, ->] (10,1.5) -- (11,1.5);
	\draw (6.75,-2) node{Family 5};

	\draw (.2,0) -- (1.2,0);
	\draw (0,-.2) arc (180:360:.5cm and .3cm);
	\draw (0,-.2) -- (0,.8);
	\draw (1,.2) -- (1,2.5);
	\draw (1,2.5) arc (180:90:.2cm);
	\draw[rounded corners = 1mm] (1.2,2.7) -- (1.5,2.7) -- (2,2);
	\draw (1.2,1) arc (-90:0:.2cm);
	\draw[rounded corners = 1mm] (1.4,1.2) -- (1.4,1.5) -- (2,2);
	\draw (1.2,0) arc (90:-90:.3cm and .5cm);
	\draw[rounded corners = 1mm] (.8,1) -- (-.4,1) -- (-1.4,0) -- (-1.4,-.2);
	\draw (-1.4,-.2) arc (-180:-90:.8cm);
	\draw (-.6,-1) -- (1.2,-1);
	\draw[rounded corners = 1mm] (-.2,0) -- (-.4,0) -- (-.8,.4);
	\draw[rounded corners  = 1mm] (-1,.6) -- (-1.4,1) -- (-1.4,3.2);
	\draw (-1.4,3.2) arc (180:90:.2cm);
	\draw (-1.2, 3.4) -- (3,3.4);
	\draw[rounded corners = 1mm] (2,2) -- (2.5,2.5) -- (2.5,2.8);
	\draw (2.5,2.8) arc (0:90:.2cm);
	\draw (2.3,3) -- (.2,3);
	\draw (.2,3) arc (90:180:.2cm);
	\draw (0,2.8) -- (0,1.2);
	\draw[rounded corners = 1mm] (2,2) -- (2.5,1.5) -- (3,1.5);
	\draw (3,1.5) arc (-90:0:.2cm);
	\draw (3.2,1.7) -- (3.2,3.2);
	\draw (3.2,3.2) arc (0:90:.2cm);

	\fill[white] (2,2) circle (.6cm);
	\draw (2,2) circle (.6cm);
	\draw (2,2) node{\tiny{$\vec{a}$}};
	\draw (1.7,2.3) node{\tiny{$-$}};
	\draw (1.7,1.7) node{\tiny{$+$}};
	\draw (2.3,2.3) node{\tiny{$+$}};
	\draw (2.3,1.7) node{\tiny{$-$}};
	
	
	\begin{scope}[xshift = 6cm]
	
		\draw (1,-.2) -- (1,2.5);
		\draw (1,2.5) arc (180:90:.2cm);
		\draw[rounded corners = 1mm] (1.2,2.7) -- (1.5,2.7) -- (2,2);
		\draw (0,-.2) arc (180:360:.5cm and .3cm);
		\draw (0,-.2) -- (0,2.8);
		\draw[rounded corners = 1mm] (2,2) -- (2.5,2.5) -- (2.5,2.8);
		\draw (2.5,2.8) arc (0:90:.2cm);
		\draw (2.3,3) -- (.2,3);
		\draw (.2,3) arc (90:180:.2cm);
		\draw[rounded corners = 1mm] (1.4,1.2) -- (1.4,1.5) -- (2,2);
		\draw (1.2,1) arc (-90:0:.2cm);
		\draw (.8,1) -- (.2,1);
		\draw (-.2,1) -- (-.4,1);
		\draw (-.4,1) arc (-90:-180:.2cm);
		\draw (-.6,1.2) -- (-.6,3.2);
		\draw (-.6,3.2) arc (180:90:.2cm);
		\draw (-.4,3.4) -- (3,3.4);
		\draw[rounded corners = 1mm] (2,2) -- (2.5,1.5) -- (3,1.5);
		\draw (3,1.5) arc (-90:0:.2cm);
		\draw (3.2,1.7) -- (3.2,3.2);
		\draw (3.2,3.2) arc (0:90:.2cm);
	
		\fill[white] (2,2) circle (.6cm);
		\draw (2,2) circle (.6cm);
		\draw (2,2) node{\tiny{$\vec{a}$}};
		\draw (1.7,2.3) node{\tiny{$-$}};
		\draw (1.7,1.7) node{\tiny{$+$}};
		\draw (2.3,2.3) node{\tiny{$+$}};
		\draw (2.3,1.7) node{\tiny{$-$}};

	\end{scope}
	
	\begin{scope}[xshift = 11cm]
	
		\draw (2,2.5) ellipse (.6cm and .4cm);
		\draw (2,1.5) ellipse (.6cm and .4cm);	
	
		\fill[white] (2,2) circle (.6cm);
		\draw (2,2) circle (.6cm);
		\draw (2,2) node{\tiny{$\vec{a}$}};
		\draw (1.7,2.3) node{\tiny{$-$}};
		\draw (1.7,1.7) node{\tiny{$+$}};
		\draw (2.3,2.3) node{\tiny{$+$}};
		\draw (2.3,1.7) node{\tiny{$-$}};
	
	\end{scope}

\end{scope}


\begin{scope}[yshift = -14cm]

	\draw[ultra thick, ->] (4.4,2) -- (5.4,2);
	\draw[ultra thick, ->] (10.5,2) -- (11.5,2);
	\draw (7.45,-1.5) node{Family 6};

	\draw (1,2) rectangle (2.5,3);
	\draw (1.75,2.5) node {\tiny{$a$}};
	\draw (1.2,2.75) node{\tiny{$-$}};
	\draw (1.2,2.25) node{\tiny{$+$}};
	\draw (2.3,2.75) node{\tiny{$+$}};
	\draw (2.3,2.25) node{\tiny{$-$}};
	\draw (1,2.25) arc (90:180:.125);
	\draw (.875,2.125) arc (0:-90:.125);
	\draw (1,2.75) arc (-90:-180:.125);
	\draw (.875,2.875) arc (0:90:.125);
	\draw[rounded corners = 1mm] (.75,3) -- (-.4,3) -- (-1.4,2) -- (-1.4,-.2);
	\draw (.75,2) -- (.2,2);
	\draw (0,2.8) -- (0,.2);
	\draw[rounded corners = 1mm] (-.2,2) -- (-.4,2) -- (-.8,2.4);
	\draw (2.8,1.7) rectangle (3.8,.2);
	\draw (3.3,.95) node{\tiny{$b$}};
	\draw (3.05,1.5) node{\tiny{$+$}};
	\draw (3.55,1.5) node{\tiny{$-$}};
	\draw (3.05,.4) node{\tiny{$-$}};
	\draw (3.55,.4) node{\tiny{$+$}};
	\draw (3.05,1.7) arc (0:90:.55cm);
	\draw (3.05,.2) arc (0:-90:.25cm);
	\draw (2.8,-.05) -- (.25,-.05);
	\draw (.25,-.05) arc (-90:-180:.25cm);
	\draw (3.55,.2) -- (3.55,-.2);
	\draw (3.55,-.2) arc (0:-90:.2cm);
	\draw (-1.4,-.2) arc (-180:-90:.2cm);
	\draw (-1.2,-.4) -- (3.35,-.4);
	\draw (0,3.2) -- (0, 3.6);
	\draw (0,3.6) arc (180:90:.2cm);
	\draw (.2,3.8) -- (3.35,3.8);
	\draw (3.35,3.8) arc (90:0:.2cm);
	\draw (3.55,3.6) -- (3.55,1.7);
	\draw (2.5,2.75) arc (-90:0:.25cm);
	\draw (2.75,3) -- (2.75,3.6);
	\draw (2.75, 4) -- (2.75,4.2);
	\draw (2.75,4.2) arc (0:90:.2cm);
	\draw (2.55,4.4) -- (-1.2,4.4);
	\draw (-1.2,4.4) arc (90:180:.2cm);
	\draw[rounded corners = 1mm] (-1.4,4.2) -- (-1.4,3) -- (-1,2.6);
	
	
	\begin{scope}[xshift = 6.5cm]
	
	\draw (1,2) rectangle (2.5,3);
	\draw (1.75,2.5) node {\tiny{$a{+}1$}};
	\draw (1.2,2.75) node{\tiny{$-$}};
	\draw (1.2,2.25) node{\tiny{$+$}};
	\draw (2.3,2.75) node{\tiny{$+$}};
	\draw (2.3,2.25) node{\tiny{$-$}};
	\draw (1,2.25) arc (90:180:.25cm);
	\draw (.75,2) -- (.75,.2);
	\draw (.75,.2) arc (180:270:.25cm);
	\draw (1,2.75) -- (.2,2.75);
	\draw (-.25,2.75) -- (-.15,2.75);
	\draw (-.25,2.75) arc (-90:-180:.25cm);
	\draw (-.5,3) -- (-.5,4.2);
	\draw (-.5,4.2) arc (180:90:.2cm);
	\draw (0,3.6) -- (0,-.2);
	\draw (0,-.2) arc (180:270:.2cm);
	\draw (2.8,1.7) rectangle (3.8,.2);
	\draw (3.3,.95) node{\tiny{$b{-}1$}};
	\draw (3.05,1.5) node{\tiny{$+$}};
	\draw (3.55,1.5) node{\tiny{$-$}};
	\draw (3.05,.4) node{\tiny{$-$}};
	\draw (3.55,.4) node{\tiny{$+$}};
	\draw (3.05,1.7) arc (0:90:.55cm);
	\draw (3.05,.2) arc (0:-90:.25cm);
	\draw (2.8,-.05) -- (1,-.05);
	\draw (3.55,.2) -- (3.55,-.2);
	\draw (3.55,-.2) arc (0:-90:.2cm);
	\draw (.2,-.4) -- (3.35,-.4);
	\draw (0,3.6) arc (180:90:.2cm);
	\draw (.2,3.8) -- (3.35,3.8);
	\draw (3.35,3.8) arc (90:0:.2cm);
	\draw (3.55,3.6) -- (3.55,1.7);
	\draw (2.5,2.75) arc (-90:0:.25cm);
	\draw (2.75,3) -- (2.75,3.6);
	\draw (2.75, 4) -- (2.75,4.2);
	\draw (2.75,4.2) arc (0:90:.2cm);
	\draw (2.55,4.4) -- (-.3,4.4);
		
	\end{scope}
	
	
	\begin{scope}[xshift = 12cm]
	
	\draw (1,2) rectangle (2.5,3);
	\draw (1.75,2.5) node {\tiny{$a{+}1$}};
	\draw (1.2,2.75) node{\tiny{$-$}};
	\draw (1.2,2.25) node{\tiny{$+$}};
	\draw (2.3,2.75) node{\tiny{$+$}};
	\draw (2.3,2.25) node{\tiny{$-$}};
	\draw (1,2.25) arc (90:180:.25cm);
	\draw (.75,2) -- (.75,.2);
	\draw (.75,.2) arc (180:270:.25cm);
	\draw (1,2.75) -- (.2,2.75);
	\draw (-.25,2.75) -- (0,2.75);
	\draw (-.25,2.75) arc (-90:-180:.25cm);
	\draw (-.5,3) -- (-.5,4.2);
	\draw (-.5,4.2) arc (180:90:.2cm);
	\draw (2.8,1.7) rectangle (3.8,.2);
	\draw (3.3,.95) node{\tiny{$b{-}1$}};
	\draw (3.05,1.5) node{\tiny{$+$}};
	\draw (3.55,1.5) node{\tiny{$-$}};
	\draw (3.05,.4) node{\tiny{$-$}};
	\draw (3.55,.4) node{\tiny{$+$}};
	\draw (3.05,1.7) arc (0:90:.55cm);
	\draw (3.05,.2) arc (0:-90:.25cm);
	\draw (2.8,-.05) -- (1,-.05);
	
	\draw (3.55,.2) arc (180:270:.2cm);
	\draw (3.75,0) -- (4,0);
	\draw (4,0) arc (-90:0:.2cm);
	\draw (4.2,.2) -- (4.2,1.7);
	\draw (4.2,1.7) arc (0:90:.2cm);
	\draw (4,1.9) -- (3.75,1.9);
	\draw (3.75,1.9) arc (90:180:.2cm);
	\draw (-.15,2.75) -- (.2,2.75);
	\draw (2.5,2.75) arc (-90:0:.25cm);
	\draw (2.75,3) -- (2.75,4.2);
	\draw (2.75,4.2) arc (0:90:.2cm);
	\draw (2.55,4.4) -- (-.3,4.4);
		
	\end{scope}

\end{scope}


\begin{scope}[yshift = -22cm,xshift = -2cm]

	\draw[ultra thick, ->] (5.8,3) -- (6.8,3);
	\draw[ultra thick, ->] (13.7,3) -- (14.7,3);
	\draw (9.45,-1.5) node{Family 7};

	\draw (1,2) rectangle (2.5,3);
	\draw (1.75,2.5) node {\tiny{$a$}};
	\draw (1.2,2.75) node{\tiny{$-$}};
	\draw (1.2,2.25) node{\tiny{$+$}};
	\draw (2.3,2.75) node{\tiny{$+$}};
	\draw (2.3,2.25) node{\tiny{$-$}};
	\draw (1,2.25) arc (90:180:.125);
	\draw (.875,2.125) arc (0:-90:.125);
	\draw (1,2.75) arc (-90:-180:.125);
	\draw (.875,2.875) arc (0:90:.125);
	\draw[rounded corners = 1mm] (.75,3) -- (-.4,3) -- (-1.4,2) -- (-1.4,.2);
	\draw (.75,2) -- (.2,2);
	\draw (0,2.8) -- (0,1.2);
	\draw (0,1.2) arc (180:270:.2cm);
	\draw (.2,1) -- (2.55,1);
	\draw (2.95,1) -- (3.2,1);
	\draw (3.2,1) arc (-90:0:.2cm);
	\draw (3.4,1.2) -- (3.4,4);
	
	\draw[rounded corners = 1mm] (-.2,2) -- (-.4,2) -- (-.8,2.4);

	\draw (2.5, 2.25) arc (90:0:.25cm);
	\draw (2.75,2) -- (2.75,.2);
	\draw (2.75,.2) arc (0:-90:.2cm);
	\draw (2.55,0) -- (-1.2,0);
	\draw (-1.2,0) arc (-90:-180:.2cm);

	\draw (0,3.2) -- (0, 4.4);
	\draw (0,4.4) arc (180:90:.2cm);
	\draw (2.5,2.75) -- (3.2,2.75);

	\draw[rounded corners = 1mm] (-1.4,5) -- (-1.4,3) -- (-1,2.6);
	\draw (-1.4,5) arc (180:90:.2cm);
	\draw (-1.2,5.2) -- (5.2,5.2);
	\draw (5.2,5.2) arc (90:0:.2cm);
	
	\draw[rounded corners = 1mm] (3.6, 2.75) -- (3.8,2.75) -- (4.5,3.5);
	\draw (3.4,4) arc (180:90:.3cm);
	\draw[rounded corners = 1mm] (3.7,4.3) -- (3.9,4.3) -- (4.5,3.5);
	\draw[rounded corners = 1mm] (4.5,3.5) -- (4.8,3) -- (5.2,3);
	\draw (5.2,3) arc (-90:0:.2cm);
	\draw (5.4,3.2) -- (5.4,5);
	
	\draw [rounded corners = 1mm] (4.5,3.5) -- (5.1,4.1) --(5.1,4.4);
	\draw (5.1,4.4) arc (0:90:.2cm);
	\draw (4.9,4.6) -- (.2,4.6);
	
	\fill[white] (4.5,3.5) circle (.6cm);
	\draw (4.5,3.5) circle (.6cm);
	\draw (4.5,3.5) node{\tiny{$\vec{b}$}};
	\draw (4.2,3.8) node{\tiny{$-$}};
	\draw (4.2,3.2) node{\tiny{$+$}};
	\draw (4.8,3.8) node{\tiny{$+$}};
	\draw (4.8,3.2) node{\tiny{$-$}};
	
	
	\begin{scope}[xshift = 8cm]

	\draw (1,2) rectangle (2.5,3);
	\draw (1.75,2.5) node {\tiny{$a{+}1$}};
	\draw (1.2,2.75) node{\tiny{$-$}};
	\draw (1.2,2.25) node{\tiny{$+$}};
	\draw (2.3,2.75) node{\tiny{$+$}};
	\draw (2.3,2.25) node{\tiny{$-$}};
	\draw (1,2.25) arc (90:180:.2);
	\draw (.8,2.05) -- (.8,1.8);
	\draw (.8,1.8) arc (180:270:.2cm);
	\draw (1,1.6) -- (2.5,1.6);
	\draw (2.5,1.6) arc (-90:0:.2cm);

	\draw (0,4.4) -- (0,1.2);
	\draw (0,1.2) arc (180:270:.2cm);
	\draw (.2,1) -- (3.2,1);
	\draw (3.2,1) arc (-90:0:.2cm);
	\draw (3.4,1.2) -- (3.4,4);
	\draw (2.5, 2.25) arc (90:0:.2cm);
	\draw (2.7,2.05) -- (2.7,1.8);
	\draw (0,4.4) arc (180:90:.2cm);
	\draw (2.5,2.75) -- (3.2,2.75);
	\draw (-.6,5) -- ( -.6,2.95);
	\draw (-.6,2.95) arc (180:270:.2cm);
	\draw (-.6,5) arc (180:90:.2cm);
	\draw (-.4,5.2) -- (5.2,5.2);
	\draw (5.2,5.2) arc (90:0:.2cm);
	\draw (-.4,2.75) -- (-.2,2.75);
	\draw (.2,2.75) -- (1,2.75);
	
	\draw[rounded corners = 1mm] (3.6, 2.75) -- (3.8,2.75) -- (4.5,3.5);
	\draw (3.4,4) arc (180:90:.3cm);
	\draw[rounded corners = 1mm] (3.7,4.3) -- (3.9,4.3) -- (4.5,3.5);
	\draw[rounded corners = 1mm] (4.5,3.5) -- (4.8,3) -- (5.2,3);
	\draw (5.2,3) arc (-90:0:.2cm);
	\draw (5.4,3.2) -- (5.4,5);
	
	\draw [rounded corners = 1mm] (4.5,3.5) -- (5.1,4.1) --(5.1,4.4);
	\draw (5.1,4.4) arc (0:90:.2cm);
	\draw (4.9,4.6) -- (.2,4.6);
	
	\fill[white] (4.5,3.5) circle (.6cm);
	\draw (4.5,3.5) circle (.6cm);
	\draw (4.5,3.5) node{\tiny{$\vec{b}$}};
	\draw (4.2,3.8) node{\tiny{$-$}};
	\draw (4.2,3.2) node{\tiny{$+$}};
	\draw (4.8,3.8) node{\tiny{$+$}};
	\draw (4.8,3.2) node{\tiny{$-$}};

	\end{scope}
	
	
	\begin{scope}[xshift = 15cm]

	\draw (1,2) rectangle (2.5,3);
	\draw (1.75,2.5) node {\tiny{$a{+}1$}};
	\draw (1.2,2.75) node{\tiny{$-$}};
	\draw (1.2,2.25) node{\tiny{$+$}};
	\draw (2.3,2.75) node{\tiny{$+$}};
	\draw (2.3,2.25) node{\tiny{$-$}};
	\draw (1,2.25) arc (90:180:.2);
	\draw (.8,2.05) -- (.8,1.8);
	\draw (.8,1.8) arc (180:270:.2cm);
	\draw (1,1.6) -- (2.5,1.6);
	\draw (2.5,1.6) arc (-90:0:.2cm);

	\draw (2.5, 2.25) arc (90:0:.2cm);
	\draw (2.7,2.05) -- (2.7,1.8);

	\draw (2.5,2.75) -- (3.2,2.75);
	\draw (.2,5.2) -- (5.2,5.2);
	\draw (5.2,5.2) arc (90:0:.2cm);
	\draw (.2,2.75) -- (1,2.75);
	\draw (.2,2.75) arc (-90:-180:.2cm);
	\draw (0,2.95) -- (0,5);
	\draw (0,5) arc (180:90:.2cm);
	
	\draw[rounded corners = 1mm] (3.2, 2.75) -- (3.8,2.75) -- (4.5,3.5);

	\draw[rounded corners = 1mm] (4.5,3.5) -- (4.8,3) -- (5.2,3);
	\draw (5.2,3) arc (-90:0:.2cm);
	\draw (5.4,3.2) -- (5.4,5);

	\draw (4.5,4.1) ellipse (.7cm and .4cm);
	
	\fill[white] (4.5,3.5) circle (.6cm);
	\draw (4.5,3.5) circle (.6cm);
	\draw (4.5,3.5) node{\tiny{$\vec{b}$}};
	\draw (4.2,3.8) node{\tiny{$-$}};
	\draw (4.2,3.2) node{\tiny{$+$}};
	\draw (4.8,3.8) node{\tiny{$+$}};
	\draw (4.8,3.2) node{\tiny{$-$}};

	\end{scope}

\end{scope}

\end{tikzpicture}$$
\caption{The diagrams on the left have $\overline{G}$ from Families 4, 5, 6, and 7 as checkerboard graphs. In each case, an isotopy yields an alternating link.}
\label{figure:Family4567}
\end{figure}

We conclude the paper with the proof of Theorem \ref{theorem:AANontrivial}, which follows from Theorems \ref{theorem:Sign} and \ref{theorem:AAJones} and Lemmas \ref{lemma:F1-3} and \ref{lemma:F4-7}.

\begin{proof}[Proof of Theorem \ref{theorem:AANontrivial}]
Let $D$ be a strongly reduced almost alternating diagram of the link $L$ such that $D$ has the fewest number of crossings among all almost alternating diagrams of $L$. Suppose that the Jones polynomial of $L$ is 
$$V_L(t) = a_0 t^{k} + a_1 t^{k+1} + \cdots + a_{n-1} t^{k+n-1}+a_n t^{k+n},$$
where $a_0$ and $a_n$ are nonzero. First, suppose that the number of components $\ell$ of $L$ is at least two. The product of the first two coefficients and the product of the last two coefficients of $t^k\left(-t^{\frac{1}{2}}-t^{-\frac{1}{2}}\right)^{\ell - 1}$ are strictly positive. However, Theorem \ref{theorem:Sign} states that at least one of the products $a_0a_1$ or $a_{n-1}a_n$ is at most zero. Thus
$$V_L(t)\neq t^k\left(-t^{\frac{1}{2}}-t^{-\frac{1}{2}}\right)^{\ell - 1}.$$ 

Now suppose that $\ell=1$, i.e. that $L$ is a knot. We need to show that $V_L(t)\neq t^k$ for some $k\in \mathbb{Z}$. Adopting the notation of Theorem \ref{theorem:AAJones}, Lemma \ref{lemma:Dual} implies that either $P$ or $\overline{P}$ is in $\{0,2\}$. Without loss of generality assume $\overline{P}\in\{0,2\}$. Thus $\alpha_{c-3}=\pm 1$.  If any of $\alpha_0$, $\alpha_1$, or $\alpha_{c-4}$ are nonzero, then at least two coefficients of the Kauffman bracket of $D$ are nonzero, and  hence $V_L(t) \neq t^k$. 

Suppose that $\alpha_0 = \alpha_1 = \alpha_{c-4}=0$. Since $\alpha_0=0$, it follows that $P=1$. If $\overline{P}=0$, then Lemma \ref{lemma:F1-3} implies $L$ is either semi-adequate or mutant to a semi-adequate link. Theorem \ref{theorem:Semi-adequate} then implies that $V_L(t)\neq t^k$ for any $k\in\mathbb{Z}$. If $\overline{P}=2$, then Lemma \ref{lemma:F4-7} implies that $L$ is alternating, rather than almost alternating. Hence the $\overline{P}=2$ case can be discarded from consideration.

Now suppose that $L$ is Turaev genus one. Theorem \ref{theorem:mutant} implies that $L$ is mutant to an almost alternating link. The result follows from the fact that the Jones polynomial does not change under mutation.
\end{proof}

\clearpage

\bibliography{aa}{}
\bibliographystyle {amsalpha}

\end{document}